\documentclass[a4paper,11pt]{amsart} 
\usepackage[left=3cm, right=3cm]{geometry}
\title[Moduli of finite flat torsors over nodal curves]{Moduli of finite flat torsors over nodal curves} 
\author[S. Mehidi, T. Poiret]{Sara Mehidi, Thibault Poiret}

\usepackage{amsmath, amssymb, mathrsfs, amsthm, shorttoc, stmaryrd, tikz-cd, graphicx}
\usepackage{bbm}
\usepackage{mathtools} 
\usepackage{comment}
\usepackage{hyperref}
\usepackage[all]{xy}
\usepackage[thinc ]{esdiff}
\usepackage{babel}

\DeclareRobustCommand{\SkipTocEntry}[5]{}

\makeatletter
\providecommand{\leftsquigarrow}{%
  \mathrel{\mathpalette\reflect@squig\relax}%
}
\newcommand{\reflect@squig}[2]{%
  \reflectbox{$\m@th#1\rightsquigarrow$}%
}
\makeatother
\let\oldtocsection=\tocsection
 
\let\oldtocsubsection=\tocsubsection
 
\let\oldtocsubsubsection=\tocsubsubsection
 
\renewcommand{\tocsection}[2]{\hspace{0em}\oldtocsection{#1}{#2}}
\renewcommand{\tocsubsection}[2]{\hspace{1em}\oldtocsubsection{#1}{#2}}
\renewcommand{\tocsubsubsection}[2]{\hspace{2em}\oldtocsubsubsection{#1}{#2}}

\usepackage{enumitem}

\usepackage[capitalize]{cleveref}
\crefformat{equation}{(#2#1#3)}

\AtBeginDocument{\renewcommand{\ref}[1]{\cref{#1}}}
\numberwithin{equation}{subsection}

\usepackage{pgf,tikz}
\usetikzlibrary{matrix, calc, arrows}
\usepackage{tikz-cd} 
\usetikzlibrary{decorations.pathmorphing,shapes}

\makeatletter
\newcommand*{\doublerightarrow}[2]{\mathrel{
  \settowidth{\@tempdima}{$\scriptstyle#1$}
  \settowidth{\@tempdimb}{$\scriptstyle#2$}
  \ifdim\@tempdimb>\@tempdima \@tempdima=\@tempdimb\fi
  \mathop{\vcenter{
    \offinterlineskip\ialign{\hbox to\dimexpr\@tempdima+1em{##}\cr
    \rightarrowfill\cr\noalign{\kern.5ex}
    \rightarrowfill\cr}}}\limits^{\!#1}_{\!#2}}}
\newcommand*{\triplerightarrow}[1]{\mathrel{
  \settowidth{\@tempdima}{$\scriptstyle#1$}
  \mathop{\vcenter{
    \offinterlineskip\ialign{\hbox to\dimexpr\@tempdima+1em{##}\cr
    \rightarrowfill\cr\noalign{\kern.5ex}
    \rightarrowfill\cr\noalign{\kern.5ex}
    \rightarrowfill\cr}}}\limits^{\!#1}}}
\makeatother

\newcommand{\on}[1]{\operatorname{#1}}
\newcommand{\bb}[1]{{\mathbb{#1}}}

\newcommand{\ca}[1]{{\mathcal{#1}}}
\newcommand{\bd}[1]{{\mathbf{#1}}}

\newcommand{\ul}[1]{{\underline{#1}}}

\def\sheafhom{\mathcal{H}om}

\def\sheafext{\mathcal{E}xt}

\newcommand{\cat}[1]{\bd{#1}}

\newcommand{\lra}{\longrightarrow}
\newcommand{\hra}{\hookrightarrow}

\newcommand{\iso}{\stackrel{\sim}{\lra}}

\DeclareMathOperator{\Star}{Star}

\def\:{\colon}
\def\.{,\dots,}

\def\et{\mathrm{\acute{e}t}}

\def\Trojac{\operatorname{TroJac}}
\def\Logjac{\operatorname{LogJac}}

\def\Pic{\operatorname{Pic}}

\def\o#1{\overline{#1}}

\newcommand{\Hom}{\operatorname{Hom}}
\newcommand{\Spec}{\operatorname{Spec}}

\theoremstyle{definition}
\newtheorem{definition}{Definition}[section]

\newtheorem{fact}[definition]{Fact}

\theoremstyle{plain}

\newtheorem{proposition}[definition]{Proposition}
\newtheorem{lemma}[definition]{Lemma}
\newtheorem{theorem}[definition]{Theorem}
\newtheorem{corollary}[definition]{Corollary}

\theoremstyle{remark}

\newtheorem{remark}[definition]{Remark}

\newtheorem{example}[definition]{Example}

\renewcommand{\phi}{\varphi} 

\date{\today}

\newcounter{nootje}
\begin{document}

\begin{abstract}
We show that log flat torsors over a family $X/S$ of nodal curves under a finite flat commutative group scheme $G/S$ are classified by maps from the Cartier dual of $G$ to the log Jacobian of $X$. We deduce that fppf torsors on the smooth fibres of $X/S$ can be extended to global log flat torsors under some regularity hypotheses.
\end{abstract}

\maketitle
	


\tableofcontents

\section{Introduction}

\noindent\textbf{1. The Albanese property.}
Curves with a section map to their Jacobian varieties, giving rise to a pullback map from finite \'etale covers of the Jacobian to finite \'etale covers of the curve. Such covers of the Jacobian are abelian (cf. \cite[Remark 15.3]{Milne1986Ab} and \cite[\S 18 p 167]{Mumford}), i.e. are torsors under the action of a finite abelian group. All abelian covers of curves are pulled back from the Jacobian, see \cite[\S 9]{Milne1986}.\\

In arithmetic geometry, it is natural to consider torsors on families of varieties under finite groups which are not necessarily constant. In this paper, we study such torsors in the fppf topology, on families of smooth curves and their degenerations. Smooth curves naturally degenerate to prestable curves (\ref{definition:prestable}), which admit nodal singularities. Allowing these nodal singularities is crucial for arithmetic applications (there does not exist a smooth curve of positive genus over $\on{Spec}\mathbb Z$) and applications to intersection theory on moduli spaces of curves (the moduli of stable curves is proper, while the moduli of smooth curves is not). Let $X \to S$ be a family of prestable curves and $G$ a finite, flat and commutative $S$-group scheme. The following result is a special case of a theorem of Raynaud (\cite[Proposition 6.2.1]{Raynaud1970Specialisation-}).

\begin{theorem}\label{Rynaud-fppf}
  There is a natural isomorphism\footnote{The notation $\on{Pic}^0$ is introduced in the section ``Notation".}
 \begin{equation}\label{Raynaud_fppf}
H^1_{fppf}(X,G)/H^1_{fppf}(S,G)
\iso \operatorname{Hom}(G^D, \on{Pic}^0_{X/S}).
 \end{equation}

\end{theorem}

Here $G^D$, $H^1_{fppf}(X,G)$ and $H^1_{fppf}(S,G)$ denote the Cartier dual $\ca{H}om(G,\bb G_{m,S})$ of $G$ and the groups of isomorphism classes of fppf $G$-torsors on $X,S$ respectively. Explicitly, the map \ref{Raynaud_fppf} is obtained as follows: given a relative fppf $G$-torsor $T$ on $X/S$ and a section $f \in G^D(S)=\on{Hom}(G,\bb G_m)$, one naturally obtains a multidegree $0$  $\bb G_m$-torsor (equivalently, line bundle) $f_*T$ on $X/S$ by change of structure group along $f$.\\

By \ref{Raynaud_fppf}, good properties of $\on{Pic}^0$ translate into good properties of fppf $G$-torsors. For example, when the curve $X/S$ is smooth, $\on{Pic}^0_{X/S}$ is the Jacobian of $X/S$, an abelian variety. In particular, it is proper over $S$, which allows one to deduce from \ref{Raynaud_fppf} results about extending relative fppf torsors on $X/S$ from a dense open of $S$ to all of $S$.\\

On the other hand, $\Pic^0$ of a nodal curve is not proper in general. Worse, if the restriction of $X/S$ to some dense open $U \subset S$ is smooth, in general $\on{Pic}^0_{X_U/U}$ cannot be extended to a proper group object over $S$-schemes. Various useful theorems (e.g. extension of relative fppf torsors) which hold over the moduli space $\ca M_g$ of smooth curves fail to extend to the boundary of the Deligne-Mumford space $\o{\ca M}_g$ of stable curves for that reason.\\

\noindent\textbf{2. Log torsors.} Logarithmic geometry was introduced in the late 1980s by Fontaine, Illusie, Kato, and others. It has found numerous applications in moduli theory, degeneration and enumerative problems, as many non-proper moduli spaces, such as the universal Jacobian over $\ca M_g$, admit natural logarithmic compactifications.\\

Prestable curves can be equipped with natural \emph{logarithmic structures}, making them satisfy the analogue of smoothness in the category of logarithmic schemes. A \emph{log curve} is a family of prestable curves equipped with such logarithmic structures. In \cite{molcho_wise_2022}, Molcho and Wise introduce the \emph{logarithmic Picard group} $\on{LogPic}_{X/S}$ of a log curve $X/S$, and show that its degree $0$ part $\on{LogJac}_{X/S}$ is a proper group model of $\on{Pic}^0_{X_U/U}$.

Compactifications over $S$ of $\on{Pic}^0_{X_U/U}$ and of its torsion $\on{Pic}^0_{X_U/U}[n]$ are often used to prove results about Picard groups of smooth curves themselves via degeneration. Degeneration techniques become more powerful as more properties of $\on{Pic}^0$ and $\on{Pic}^0[n]$ are preserved in the compactification. Log Jacobians and their torsion subgroups are the only known such compactifications which remain modular groups. This makes them a powerful tool in algebraic and arithmetic geometry. For example, they are used in \cite{Kajiwara2013Logarithmic-abe} to provide an easier construction of the Deligne-Rapoport compactifications of the modular curves $Y(N),Y_0(N),Y_1(N)$.\\

Our first result is a proof that $\on{LogJac}_{X/S}$ parametrizes $G$-torsors in a topology which arises naturally in log geometry, the \emph{Kummer log flat} (klf) topology. Rougly speaking, these are torsors which are allowed to tamely ramify along a well-chosen divisor. Specifically, we show

\begin{theorem}[\ref{mainrslt}]\label{thm1}
After equipping $X/S$ with a structure of log curve, there is a canonical isomorphism\footnote{Here, $H^1_{klf}$ denotes the group of isomorphism classes of Kummer log flat torsors.}
 \[
 H^1_{klf}(X_{},G_{})/H^1_{klf}(S,G) \simeq   {\Hom}(G_{}^D, \mathrm{LogJac}_{X/S}).
 \]
\end{theorem}
Since $\on{LogJac}_{X/S}$ is proper, this suggests that the good properties of fppf torsors on smooth curves which follow from \ref{Rynaud-fppf} and the properness of the Jacobian might extend to klf torsors on prestable curves. One such property is \emph{extension of torsors}.\\

\noindent\textbf{3. Extending torsors.} Extension of torsors deals with the following question. If $U,S,G$ are as above, $Y/S$ is a flat morphism of finite type and $T$ is a $G$-torsor on $Y\times_S U \to U$, does there exist an fppf $G$-torsor on $Y/S$ which restricts to $T$? Is such a torsor unique if it exists?\\

This question seems to first appear in work of Grothendieck and Raynaud in the early 1960s (\cite[Exposé X]{GrothRev}, see also \cite[Theorem 5.7.10]{FundamentalGp}). Using the theory of specialization of the \'etale fundamental group, they prove the existence and uniqueness of extensions under various hypotheses on $Y,S,G$, including smoothness and properness of $Y/S$ and constancy of $G/S$.\\

The question has been extensively investigated in more general contexts since then, and many examples are known in which fppf torsors do not extend. For instance, when $G/S$ is étale, any fppf extension must be unramified, and there can be obstuctions to such extensions when $Y/S$ is singular. Logarithmic geometry allows for the consideration of torsors that admit mild ramification.\\


We may use \ref{thm1} to show that fppf torsors on the smooth fibres of prestable degenerations extend uniquely to Kummer log flat torsors on the whole family.


\begin{theorem}[\ref{maincor}]\label{thm2}
Let $\pi \colon X \to S$ be a log curve with $S$ log regular (e.g. $S$ is a toric variety and $\pi$ is smooth over the dense torus, or $S$ is regular and $\pi$ smooth over the complement of a normal crossings divisor). Let $U \subset S$ be the dense open over which $\pi$ is smooth. If $G^D \to S$ is \'etale over $U$ and tamely ramified along $S \setminus U$, then any relative fppf $G_U$-torsor over $X_U$ extends uniquely to a relative klf $G$-torsor over $X$.
\end{theorem}

When $G/S$ is constant and of invertible order, this was mentioned in \cite[Theorem 7.6]{illusie2002overview}.\\


\noindent\textbf{4. Illustrating the main theorems through examples.} In \ref{section:examples}, we discuss extension of torsors for some groups $G/S$ of interest. The results are summarized below.

\begin{proposition}\label{proposition:to_extend_or_not_to_extend_that_is_the_question}
    Let $X/S$ be a log curve with $S$ log regular. Let $U \subset S$ be the dense open over which the curve is smooth.
    \begin{enumerate}
        \item If $G= \mu_n$ or if the order of $G$ is invertible in $S$ (eg. $G= \bb Z/n \bb Z, G=\mathrm{ LogPic}_{X/S}[n]$ for $n\in \ca O_S(S)^\times$), \ref{thm2} applies and fppf $G$-torsors on $X_U/U$ extend uniquely to klf $G$-torsors on $X/S$ (\ref{ext_torsors}).
        \item Let $p$ be a prime. There exist a discrete valuation ring $R$ of mixed characteristic $(0,p)$, a log curve $X/\Spec R$ and fppf torsors $T,T'$ on the generic fibre of $X/\Spec R$ under the groups $\bb Z/p\bb Z$ and $\on{LogPic}_{X/R}[p]$ such that neither $T$ nor $T'$ extend to klf torsors on $X/\Spec R$. (cf. \ref{non-ext} and \ref{LogJac[n]}).
    \end{enumerate}
\end{proposition}

\begin{proposition}[\ref{LogJac[n]}]
    Let $X/S$ be a log curve and $n$ an integer. There exists a \emph{universal $\on{LogPic}_{X/S}[n]$-torsor $\ca T$ on $X/S$} with the following universal property. For any finite, flat and commutative group $G/S$ killed by $n$, any $G$-torsor on $X/S$ is obtained from $\ca T$ by change of structure group along a unique map $\on{LogPic}_{X/S}[n] \to G$.
\end{proposition}

We describe explicitly the universal $\on{LogPic}[n]$-torsor on a nodal degeneration of an elliptic curve in \ref{example:univ_torsor_on_nodal_cubic}.

In equicharacteristic $p$, the building blocks for finite flat commutative group schemes killed by $p$ are $\mu_p$, $\bb Z/p\bb Z$ and the group $\alpha_p$ of $p$-nilpotents inside $\bb G_a$, so we also talk about $\alpha_p$-torsors. We show:

\begin{proposition}[\ref{alpha_p}]
    Let $X/S$ be a log curve such that $p=0$ in $\ca O_S$. Then, the natural map
    \[
    H^1_{fppf}(X/S,\alpha_p) \to H^1_{klf}(X/S,\alpha_p)
    \]
    is an isomorphism.
\end{proposition}

 We conclude with a discussion about the structure of the group of $\alpha_p$-torsors on a prestable curve. We show

\begin{proposition}[\ref{proposition:structure_alpha_p_torsors}]
    Let $\pi \colon X \to \Spec k$ be a prestable curve over an algebraically closed field of characteristic $p$. Let $X^\nu \to X$ be the normalization map and $r,\ca I$ be respectively the $p$-rank and local-local part of $\on{Pic}_{X^\nu/k}[p]$ (see \ref{local_local_part_and_p_rank}). The $k$-group sheaf of relative fppf $\alpha_p$-torsors on $X/\Spec k$ is isomorphic to
    \[
    \alpha_p^{H_1} \times \alpha_p^r \times \sheafhom_k(\alpha_p,\ca I).
    \]
\end{proposition}
We discuss which groups $\sheafhom_k(\alpha_p,\ca I)$ can arise in \ref{remark:alpha_p_tors_on_nodal_curve} and \ref{proposition:hom_ap_ap_=_Ga}.\\

\noindent\textbf{5. More applications.} Our main motivation for \ref{thm1} was to obtain (klf) extension of torsors in cases where no fppf extension exists, as in \ref{thm2}. One could push this further: even when not all $G_U$-torsors on $X_U/U$ extend to log $G$-torsors on $X/S$, one might want to find another group model $G'/S$ of $G_U$ and an extension to a log $G'$-torsor instead. \ref{Group stack} suggests that this might be possible at the cost of allowing $G'$ to be a group stack. We intend to explore this direction further in future work, using the work of \cite{SBr} on dualizable group stacks.

\ref{thm1} and klf extension open the way to the use of degeneration techniques for the classification of torsors. For example, we can produce geometrically meaningful invariants for torsors on families of smooth curves as follows. Let $s \to S$ be a geometric point such that the fiber $X_s$ is nodal and $n$ be an integer which kills $G$. It follows from \ref{thm1} that klf $G$-torsors on $X/S$ are partitioned by maps
\begin{equation}\label{discrete_invariant}
    G_s^D \to \on{TroPic}_{X_s/s}[n],
\end{equation}
where the \emph{tropical Picard group} $\on{TroPic}_{X_s/s}$ is an explicit, purely combinatorial invariant defined in terms of the homology of the dual graph of $X_s/s$. This makes the maps \ref{discrete_invariant} convenient to compute and work with. In particular, when we are in a situation in which klf extension of $G$-torsors holds, fppf $G_U$-torsors on the smooth curve $X_U/U$ can be partitioned by maps \ref{discrete_invariant} as well. For one-parameter degenerations, one could already extract these invariants using \cite[Corollary 1.10]{Sara} and an analysis of the connected components of the N\'eron model of the Jacobian. The interest of \ref{thm1} lies in showing that the behavior of these invariants in larger families is explicitly controlled by the geometry of $\on{TroPic}$, hence by the monodromy action on the dual graph of $X_s/s$ when $n$ is invertible on $S$.\\

\addtocontents{toc}{\SkipTocEntry}
\section*{Acknowledgements} 

The authors would like to warmly thank Jean Gillibert, Pim Spelier, Jonathan Wise, Valentijn Karemaker, Rachel Pries, Dajano Tossici, Niels Borne, Soumya Sankar and Matthieu Romagny for useful conversations and email exchanges about this work, as well as Dhruv Ranganathan, Marta Pieropan and David Holmes for their helpful comments on an earlier version. We also thank the referee for their useful suggestions to improve the presentation of the paper.

We thank the Institut de Math\'ematiques de Bordeaux for funding a research visit of T. Poiret to Bordeaux. S. Mehidi was funded partially by the Dutch Research Council (NWO) grant VI.Vidi.213.019, which also funded a research visit of T. Poiret to Utrecht. T. Poiret was funded partially by EPSRC New Investigator Grant EP/V051830/1 and by EPSRC New Investigator award EP/X002004/1.

For the purpose of open access, a CC BY public copyright license is applied to any Author Accepted Manuscript version arising from this submission.

\section*{Notation}\label{Notation}
 Given a scheme or log scheme $S$ and a $S$-group $G$, we denote by $G^0$ the fiberwise-connected component of identity of $G$. The main cases of interest in this paper are Picard groups and log Picard groups of nodal curves, which have a natural \emph{degree map} of $S$-sheaves $\on{deg} \colon G \to \bb Z$. We then denote by $G^{\mathrm{deg}=0}$ the kernel of the degree map. When $G$ is a log Picard group, we have $G^0=G^{\deg=0}$. When $G$ is the relative Picard group of a nodal curve $\pi \colon X \to S$, $G^0$ coincides with the \emph{multidegree $0$ part of $G$}, i.e. the subsheaf parametrizing line bundles of degree $0$ on every irreducible component of every geometric fibre of $\pi$.

\section{Background}\label{background}

We begin by recalling in \ref{Raynaud} a result of Raynaud, which implies that fppf torsors over a proper curve $X$ under a finite group are parametrized by the multidegree $0$ Picard group of $X$. We are particularly interested in applications to extension of torsors on families of curves. For such applications, it is best to work with proper moduli spaces, hence with stable curves. Multidegree $0$ Picard groups of stable curves are not themselves proper: to obtain a proper group moduli space, one must work with the \emph{logarithmic Jacobian} $\mathrm{LogPic}^{0}$ instead. This raises two natural questions: is there a logarithmic analogue of Raynaud's result, and can one use it to extend torsors? We will give affirmative answers to both questions in our main results \ref{mainrslt} and \ref{maincor}. We first give the necessary background on moduli of torsors and on log geometry.


\subsection{Specialization of fppf torsors}\label{Raynaud}

\begin{definition}\label{definition:prestable}
    A \emph{prestable curve} is a proper and flat morphism $\pi \colon X \to S$ whose geometric fibres are at-worst-nodal, of pure dimension $1$ and connected. By \cite[\href{https://stacks.math.columbia.edu/tag/0E6L}{Tag 0E6L}]{stacks-project}, this implies $\pi_*\ca O_X=\ca O_S$ universally.
\end{definition}

We refer to \cite[Proposition 6.2.1]{Raynaud1970Specialisation-} for the following theorem, although Raynaud states that he only gives a proof because of the lack of a suitable reference.

\begin{theorem}\label{theorem:Raynaud}
    Let $\pi \colon X \to S$ be a proper, flat and finitely presented morphism. Let $G$ be a finite, flat, finitely presented and commutative $S$-group scheme. If $\pi_*\ca O_X=\ca O_S$ universally, then there is a canonical isomorphism
    \begin{equation}\label{eqn:Raynaud}
        H^1_{fppf}(X,G)/H^1_{fppf}(S,G) = \on{Hom}(G^D,\on{Pic}_{X/S}),
    \end{equation}
    where $G^D$ is the Cartier dual $\sheafhom_S(G,\bb G_m)$ of $G$.
\end{theorem}

\begin{remark}
    Raynaud formulates this result as
    \[
    R^1_{fppf}\pi_*\pi^*G \iso \sheafhom(G^D,\on{Pic}_{X/S}).
    \]
    Under our hypotheses the first derived pushforward $R^1_{fppf}\pi_*\pi^*G$ parametrizes relative $G$-torsors on $X/S$.
\end{remark}
When $X/S$ is a prestable curve, we deduce that $\mathrm{Pic}^{0}_{X/S}$ can serve as a parametrizing space for torsors under \emph{any} finite flat commutative $S$-group scheme (since the Cartier dual of such a group remains finite flat). When \( X/S \) is smooth, \( \mathrm{Pic}^{0}_{X/S} \) is the Jacobian of \( X/S \), which is proper, and the isomorphism above becomes particularly useful in practice.

\subsection{Prestable curves and their Picard groups}\label{Prestable}

The relative Picard group $\on{Pic}_{X/S}$ of a smooth curve $X/S$ admits a degree map to $\bb Z$ whose kernel $\on{Pic}^{deg=0}_{X/S}$ is an abelian scheme over $S$, the \emph{Jacobian}. When $X/S$ is only a prestable curve, $\on{Pic}^{deg=0}_{X/S}$ remains a smooth group over $S$, but is no longer proper. It is known that there is no smooth and proper group functor over the moduli space of prestable curves which coincides with the Jacobian over the locus of smooth curves. Finding good models of the Jacobian is a subtle problem subject to a vast literature, including but not restricted to [\cite{Ishi}, \cite{DSouza},
\cite{OdaSeshadri}, \cite{AltmanK80}, \cite{AltmanK79},\cite{Kajiwara}, \cite{Caporaso}, \cite{Pandharipande}, \cite{Jarvis}, \cite{Esteves}, \cite{CaporasoL}, \cite{CaporasoLu}, \cite{Melo}, \cite{Chiodo}].

In \cite{molcho_wise_2022}, it is shown that one \emph{can} get such a smooth and proper group functor at the cost of working in the category of \emph{log schemes}. Log schemes are schemes equipped with some extra rigidifying data. Prestable curves can be equipped with natural log scheme structures, making them \emph{log curves}. Molcho and Wise define the \emph{logarithmic Picard group} $\on{LogPic}_{X/S}$ of a log curve $X/S$ and prove

\begin{fact}
    Let $X/S$ be a log curve. Then $\on{LogPic}_{X/S}$ is a log smooth group over $S$. There is a natural map $\on{Pic}_{X/S} \to \on{LogPic}_{X/S}$, which is an isomorphism if $X/S$ is smooth. The degree map on $\on{Pic}_{X/S}$ extends to a degree map $\on{LogPic}_{X/S} \to \bb Z$, whose kernel is proper.
\end{fact}


We collect a few well-known results about the structure of prestable curves and their Picard groups in families. We put a particular emphasis on the special case of prestable curves over an affine toric variety. In this setting, log Picard groups show up naturally. All our toric varieties are normal unless specified otherwise.

\subsubsection{The failure of properness of $\on{Pic}$}

We illustrate the failure of properness of the degree $0$ part of $\on{Pic}$ for nodal curves, in a way that hints at what the log Picard group should be.

\begin{example}\label{example:nonproper_pic_of_elliptic_degeneration}

    Let $S=\on{Spec} R$ where $R$ is a discrete valuation ring with uniformizer $\pi$. Let $\eta,s \in S$ be the generic point and special point respectively. Let $X/S$ be a prestable curve, elliptic over $\eta$ with discriminant $\pi^n$, and whose fibre over $s$ is a nodal cubic (\ref{nodalcu}). 

\begin{figure}[ht]
    \begin{center}
\begin{tikzpicture}

\draw[thick] (0,0) .. controls (-1,2) and (2,2) .. (0,4);

\draw [thick, xshift=3cm, yshift=2cm] plot [smooth, tension=1] coordinates { (0.75,-2) (-0.5,0.75) (-0.5,-0.75) (0.75,2)};
    


\node at (1.5,1.5) { $\rightsquigarrow$};
\node at (0,-0.5) {$t \neq 0$};
\node at (3,-0.5) {$t=0$};

\end{tikzpicture}

\end{center}
    \caption{The family of curves $X/S$}\label{nodalcu}
\end{figure}

    Then, the local equation for the total space $X$ at the node $x\in X_s$ is $uv=\pi^n$, where $u,v$ are local parameters for the two branches of $X_s$. The dual graph $\Gamma_{X/S}$ of $X_s$ (cf. \ref{dualgr}) has one vertex and one loop corresponding to $x$, and we metrize it by giving length $n$ to the loop.

    The restriction map
    \begin{align}\label{eqn:restriction_map_pic_DVR_0}
        \Phi_{X/S} \colon \on{Pic}_{X/S}(S) \to \on{Pic}_{X/S}(\eta)
    \end{align}
    is not surjective. In particular, $\on{Pic}_{X/S}$ is not proper. We may repeatedly blow up $X$ at nodes of the special fibre until we obtain a model $X' \to X$ such that
    \begin{itemize}
        \item $X'_\eta = X_\eta$.
        \item The special fibre $X'_s$ consists of $n$ rational components $(D_i)_{i\in \bb Z/n\bb Z}$, meeting cyclically along nodes at which the local equation for $X'$ is $uv=\pi$. Reindexing if necessary, we assume that $D_0$ is the strict transform of $X_s$ in $X'_s$.
    \end{itemize}

The induced morphism of metrized dual graphs $\Gamma_{X'/S} \to \Gamma_{X/S}$ is a \emph{subdivision}, i.e. the source is obtained from the target by replacing some edges with chains of edges of the same total length (\ref{graph}).
\begin{figure}[ht]
    \begin{center}
\begin{tikzpicture}

\draw (0,0) circle(1);

\draw (6,0) circle(1);

\draw[->,thick] (2.5,0) -- (3.5,0);

\filldraw[black] (7,0) circle (3pt);

\foreach \angle in {0, 72, 144, 216, 288} {
    \filldraw[black] ({cos(\angle)}, {sin(\angle)}) circle (3pt);
}

\foreach \i in {0, 1, 2, 3, 4} {
    \pgfmathtruncatemacro{\next}{mod(\i+1,5)}
    \node at ({(cos(\i*72) + cos(\next*72))/2 * 1.7}, 
              {(sin(\i*72) + sin(\next*72))/2 * 1.7}) {1};
}

\node at (6,1.5) {5};

\end{tikzpicture}


\end{center}
    
    \caption{The subdivision $\Gamma_{X'/S} \to \Gamma_{X/S}$ for $n=5$.}\label{graph}
\end{figure}

The surface $X'$ is regular, so the closure in $X'$ of any Cartier divisor on $X_\eta$ is Cartier and $\Phi_{X'/S}$ is surjective. Still, $\on{Pic}_{X'/S}$ is not a proper model for $\on{Pic}_{X_\eta/\eta}$ since $\Phi_{X'/S}$ is not injective: its kernel is generated by the line bundles $\ca O(D_i)$.
\end{example}

    \ref{example:nonproper_pic_of_elliptic_degeneration} suggests that one might obtain a functor with the valuative criterion for properness by taking a nodal curve $X/S$ to a limit of appropriate quotients of $\on{Pic}_{X'/S}$ for a well-behaved system of blowups $X' \to X$. This is essentially what the logarithmic Picard group is, as we will see in \ref{Pic and LogPic}. In order to make this idea work formally and produce a genuine moduli functor, we need to work with \emph{log schemes}, which we will discuss in \ref{log sch etc}. Log Picard groups differ from $\on{Pic}$ in a way that is explicitly controlled by the combinatorics of dual graphs. First, we illustrate the relevant ideas in the nice setting of degenerations over affine toric varieties.

\subsubsection{Families over toric varieties}

Log schemes are essentially schemes equipped with local maps to affine toric varieties which are compatible modulo the torus action. Here, we briefly discuss families of curves and their Picard groups over affine toric varieties.

\begin{remark}\label{remark:monoids_of_TVs_are_integral_saturated}
    Not-necessarily-normal flat families of toric varieties over a scheme $S$ are obtained from such families over $\bb Z$ by base change. Let $X$ be an affine, not-necessarily-normal toric variety over $\bb Z$. Let $P$ be the monoid consisting of those characters of the torus of $X$ which extend to rational functions on $X$. Then $X \iso \on{Spec} \bb Z[P]$, the character lattice is the groupification $P^{gp}$, and $P$ is finitely generated, cancellative and torsion-free. The toric variety $X$ is normal if and only if $P$ is \emph{saturated}, i.e. $P=P\otimes_{\bb N} \bb Q_{\geq 0} \cap P^{gp}$ (\cite[Theorem 1.3.5]{cox2011toric}).
\end{remark}

From now on, if $X$ is an affine toric variety over a ring $R$, when we write $X=\on{Spec}R[P]$ we always mean that $P$ is the cancellative, torsion-free, saturated monoid of characters of the torus which extend to rational functions on $X$. Such monoids are called \emph{fine and saturated}, or \emph{fs} for short.

\begin{example}\label{example:logpic_over_affine_TV}
    Let $S=\on{Spec} \bb C[P]$ be an affine toric variety over $\on{Spec} \bb C$. Suppose that $P$ has no nontrivial units, so that the unique closed torus orbit of $S$ is a point $s$. Let $U \subset S$ be the dense torus and $X/S$ be a prestable curve, smooth over $U$. Denote by $n$ the rank of $P^{gp}$. The restriction
    \[
    \Phi_{X/S} \colon \on{Pic}_{X/S}(S) \to \on{Pic}_{X/S}(U)
    \]
    is neither surjective nor injective in general (see \ref{example:nonproper_pic_of_elliptic_degeneration}) and we have graph-theoretic combinatorial descriptions for its kernel and cokernel. Denote by $\Gamma_{X/S}$ the dual graph of $X_s$. For each edge $e$ of $\Gamma_{X/S}$, the local equation for $X$ at the corresponding node of $X_s$ is $uv=t$ for some $t\in P$ where $u,v$ are parameters for the branches of $X_s$. We metrize $\Gamma_{X/S}$ by giving length $t$ to the edge $e$.

    The kernel of $\Phi_{X/S}$ consists of line bundles attached to Cartier divisors which do not meet $X_U$. Weil divisors $D$ on $X$ which do not meet $X_U$ are in natural bijection with $P^{gp}$-valued functions $f_D$ on the vertices of $\Gamma_{X/S}$. $D$ is Cartier if and only if $f_D$ is \emph{piecewise-linear}, i.e. the values of $f_D$ at the endpoints of any edge $e$ differ by an integer multiple of the length of $e$. Thus, the kernel of $\Phi$ is the image of a natural map $\on{PL}(X/S) \to \on{Pic}_{X/S}(S)$ where $\on{PL}(X/S)$ is the group of piecewise-linear functions.
    
    The failure of surjectivity of $\Phi$ can be described combinatorially as well: we may find a diagram
    \[
    \begin{tikzcd}
        X' \arrow[r] & X_{S'} \arrow[r]\arrow[d] & S'\arrow[d] \\
        & X \arrow[r] & S
    \end{tikzcd}
    \]
    where the square is cartesian, $S' \to S$ is a toric alteration\footnote{A toric alteration is a composition of a toric blowup with a finite toric map induced by a finite-index extension of the cocharacter lattice $P^{gp,\vee}:=\on{Hom}(P,\bb Z)$.}, $X' \to X_{S'}$ is a toric blow-up and $X',S'$ are regular schemes. Then $\Phi_{X'/S'}$ is surjective so we may understand the cokernel of $\Phi_{X/S}$ in terms of the combinatorics of the subdivision of dual graphs $\Gamma_{X'/S'} \to \Gamma_{X_{S'}/S'}$ and of the extension of cocharacter lattices $\on{Cochar}(S') \to \on{Cochar}(S) = \on{Hom}(P,\bb Z)$. We omit the details, since we are about to treat the more general case where $S$ is an arbitrary log scheme.
\end{example}

\subsection{Log schemes}\label{log sch etc}

In the theory of models and degenerations, it is often natural to want to work with pairs $(\ul X,D)$ of a scheme with a divisor on it and ask that morphisms respect the distinguished divisors. If one is also interested in moduli problems, this can be inconvenient since the pull-back of a divisor is not necessarily a divisor. A more flexible approach is to work with \emph{log structures} instead.

Rougly, a log structure is the data of local maps to toric varieties defined modulo the torus action. Any divisor uniquely induces a log structure, but log structures can be pushed forward and pulled back. This often gives rise to well-behaved moduli functors of schemes with log structures, which have no well-behaved scheme-theoretic counterpart.

In this section, we recall the basics of log schemes and of the useful topologies on them. For a more comprehensive introduction, we refer to \cite{Ogus2018Lectures} and the original \cite{Kato} .

\begin{definition}
    A \emph{prelog structure} on a scheme $\ul X$ is a morphism $\alpha_X\colon M_X \to \ca O_X$ of \'etale sheaves of monoids\footnote{The monoid law on $\ca O_X$ is multiplication}. It is a \emph{log structure} if $\alpha_X^{-1}\ca O_X^\times \simeq {\ca O}_X^\times$. A prelog structure naturally induces a log structure. The category of \emph{schemes with log structure} has as objects the pairs $X=(\ul X,\alpha_X)$. A morphism $X \to Y$ is a scheme map $f \colon \ul X \to \ul Y$ together with a commutative square
\[
\begin{tikzcd}
f^*M_Y \arrow{r}{f^*\alpha_Y} \arrow[swap]{d}{} & f^*\mathcal{O}_Y\arrow{d}{} \\
M_X  \arrow{r}{\alpha_X} & \mathcal{O}_X
\end{tikzcd}
\]
of sheaves of monoids on $\ul X$. We often write $X=(\ul X,M_X)$ instead of $X=(\ul X,\alpha_X)$.
\end{definition}

\begin{example}\label{trivial_log}
    The \emph{trivial log structure} on a scheme $\ul X$ is $\ca O_X^\times \hra \ca O_X$. 
\end{example}

\begin{example}\label{example:divisorial_log_structure}
    Let $\ul X$ be a scheme and $D$ a divisor on $X$. The \emph{divisorial log structure} induced by $D$ on $\ul X$ is the subsheaf $\alpha \colon M_X \to \ca O_X$ consisting of rational functions which are invertible on $\ul X \backslash D$. We always implicitly equip toric varieties with the divisorial log structure from the toric boundary.
\end{example}

Recall that affine toric varieties are dual to fs monoids by \ref{remark:monoids_of_TVs_are_integral_saturated}.

\begin{example}\label{example:log_point}
    Let $\o P$ be a sharp\footnote{A monoid is called sharp when its only invertible element is $0$.} fs monoid and $k$ be a field. The \emph{log point} $(\Spec k,\o P)$ is the scheme $\Spec k$ equipped with the log structure
    \begin{align*}
   k^\times \times \o P & \to k \\  
(\lambda,m) & \mapsto \left\{
    \begin{array}{ll}
        \lambda & \mbox{if } m=0 \\
        0 & \mbox{otherwise.}
    \end{array}
\right.
\end{align*}

   
    If $k$ is separably closed, we say $(\Spec k,\o P)$ is a \emph{geometric log point}.
\end{example}

\begin{definition}
    A morphism $f\colon X \to Y$ of schemes with log structures is \emph{strict} if $f^*M_Y = M_X$. A (global) \emph{chart} for $X$ is a strict morphism to an affine toric variety over $\bb Z$. A \emph{log scheme} is a scheme with log structure which admits charts \'etale-locally. We denote by $\cat{LSch}$ the category of log schemes.
\end{definition}

\begin{remark}
    Log points are log schemes. A log point $(\on{Spec} k, \o P)$ can be seen as the origin of the affine $k$-toric variety $\on{Spec}k[\o P]$, with pullback log structure: we are forcing the point to remember some information about the local structure of the toric boundary.
\end{remark}

We will flexibly use the underline notation for either schemes or the underlying scheme functor $\cat{LSch} \to \cat{Sch}$. The log structure on a log scheme $X$ will always be denoted by $M_X \to \ca O_X$. We write $\o M_X$ for the characteristic (sheaf of) monoid(s) $M_X/\ca O_X^\times$.

\subsubsection{Topologies on log schemes}




\begin{definition}[log modifications]
    Let $X$ be a log scheme which admits a global chart $X \to \Spec \bb Z[P]$. Let $I\subset P$ be an ideal. The \emph{log blow-up} of $X$ in $I$ is the morphism $f \colon Y \to X$ obtained by blowing up $I$ inside $\Spec \bb Z[P]$, pulling back to $X$, and saturating the result. This construction is compatible with \'etale localization, so when $X$ does not necessarily have a global chart and $I\subset M_X$ is a sheaf of ideals, we define similarly the log blow-up of $X$ in $I$. Similarly, a \emph{log modification} is a morphism of log schemes which is locally pulled back from a toric modification of charts.
\end{definition}

\begin{remark}
    Log blow-ups are log modifications. Any log modification is dominated by a log blow-up as a consequence of the toric Chow lemma \cite[Lemma 6.9.2.]{Danilov1978GeometryToricVar}.
\end{remark}

\begin{definition}[root stacks]
    Let $X$ be a log scheme with a global chart $X \to \Spec \bb Z[P]$ and let $P \to Q$ be a finite index and injective monoid map. The \emph{root stack} $f \colon Y \to X$ with respect to $P \to Q$ is obtained as follows. Let $\overline Y=X \times_{\Spec \bb Z[P]} \Spec \bb Z[Q]$. Then $Y$ is the quotient stack $[\overline Y/K]$, where $K$ is the kernel of the map of tori $\Spec \bb Z[Q^{gp}] \to \Spec \bb Z[P^{gp}]$. We give $Y$ the strict log structure over $\Spec \bb Z[Q]$. This construction globalizes, so when $X$ has no global chart we may still define the root stack along an injective map $\o{M}_X \to \ca F$ of sheaves of monoids which is locally of finite index. 
\end{definition}

\begin{example}
    If $X=\Spec \bb C[x,y]$ with log structure given by $\bb C^\times x^{\bb N}$, we have a global chart $X \to \Spec \bb Z[\bb N]$ where $1\in \bb N$ maps to $x$. The root stack $Z \to X$ along the monoid map $\bb N \to \frac{1}{2}\bb N$ is the quotient stack of $\Spec \bb C[x,y,z]/(z^2-x)=\bb A^2_{y,z}$ by $\mu_2$, which acts by multiplication on the $z$ coordinate. In other words, $Z$ is obtained from $X$ by adjoining to $x$ a \emph{canonical} square root.
\end{example}

\begin{remark}
    Log modifications and root stacks are simultaneously monomorphisms and epimorphisms in $\cat{LSch}$.
\end{remark}

\begin{definition}\label{definition:topologies_on_LSch}
    We define the following topologies on $\cat{LSch}$.
    \begin{enumerate}
        \item For $\tau\in\{\text{\'etale, smooth, fppf}\}$, the \emph{strict $\tau$ topology}, which we still denote by $\tau$, consists of the strict $\tau$ coverings.
        \item The \emph{Kummer log flat} (klf) topology is generated by strict fppf coverings and root stacks.
        \item The \emph{log \'etale} topology is generated by strict \'etale coverings, log modifications, and root stacks of local index invertible on the target. The \textit{Kummer log étale} topology is generated by strict étale covererings and root stacks of local index invertible on the target. 
    \end{enumerate}
\end{definition}

\begin{remark}
    In \cite{Kato} and other standard references, the klf and log \'etale topology are defined in terms of chart criteria. The equivalence with our definition is \cite[Corollary 3.6]{illusie2002overview} for the klf topology. For the log \'etale topology, it is stated in \cite[\S 9.1]{illusie2002overview} and proved in \cite[Proposition 3.9]{Nakayama}.
\end{remark}


\begin{example}\label{Klflat_morphism}
    Consider an extension of discrete valuation rings
    \[
    A \to A[\sqrt[n]{\pi}]=:B,
    \]
    where $\pi$ is a uniformizer for $A$ and $n$ is an integer invertible in $A$. Endow both $\Spec A$ and $\Spec B$ with their divisorial log structures. The morphism $f \colon \Spec B \to \Spec A$ is Kummer log flat (actually, Kummer log étale). Indeed, $f$ has a \emph{cartesian} global chart, dual to the cocartesian square of monoids
    \begin{equation*}
        \xymatrix{
        \bb N \ar[d]^{j} \ar[r]^{g} & \bb N\ar[d]^{k} \\
        A \ar@{^{(}->}[r] & B \\}.
    \end{equation*}
    Here $j,k$ respectively map $1$ to $\pi,\sqrt[n]{\pi}$ and $g$ is multiplication by $n$. Note that the underlying scheme map $\ul f$, while not étale, is tamely ramified.
\end{example}

\subsubsection{Curves and regularity}
Smoothness and regularity have natural analogues in $\cat{LSch}$.
\begin{definition}[log smoothness, log regularity]
A morphism of log schemes is \emph{log smooth} if it satisfies the infinitesimal lifting criterion (\cite[\S 3.3]{Kato}) and is locally finitely presented. Following \cite[Definition 2.1]{KatoToric}, if $X=(\underline{X},M_X)$ is a log scheme, $x\in\ul X$ a point and $I(x,M_X)$ the ideal of $\mathcal{O}_{X,x}$ generated by $M_{X,x} \backslash \mathcal{O}_{X,x}^{\times}$, we say that $X$ is \emph{log regular at $x$} if
\begin{itemize}
    \item $\mathcal{O}_{X,x}/I(x, M_X)$ is a regular local ring.
    \item $\mathrm{dim} (\mathcal{O}_{X,x}) = \mathrm{dim} (\mathcal{O}_{X,x}/I(x, M_X)) + \mathrm{rank}(M_{X,x}^{gp}/\mathcal{O}_{X,x}^{\times})$.
\end{itemize}
We say that $X$ is \emph{log regular} if it is log regular at every point.
\end{definition}

\begin{example}
 A regular scheme endowed with the log structure induced by an snc divisor is log regular. A toric variety with log structure induced by the toric boundary is log regular.
\end{example}

\begin{remark}
     A log regular log scheme has a log \'etale covering by regular schemes whose log structure is that of an snc divisor (cf. \cite[\S 5.2]{Niziol}). This generalizes the fact that toric varieties admit resolutions of singularities via log blow-ups.
\end{remark}

\begin{remark}
    Log smoothness is preserved by base change. It can be characterized in terms of charts, in a similar way to the topologies of \ref{definition:topologies_on_LSch}. As was the case for regular schemes and smooth maps, any log scheme which admits a log smooth map to a log regular target is log regular itself.
\end{remark}

Prestable curves can be equipped with natural log structures, as follows.
\begin{definition}\label{definition:log_curve}
    A log scheme map $f \colon X \to S$ is a \emph{log curve} if
    \begin{itemize}
        \item The underlying scheme map $\ul X \to \ul S$ is a prestable curve.
        \item $f$ is strict away from the nodes.
        \item For each strict geometric point $s \to S$ and node $x\in X_s$, we have
        \[
        \o M_{X,x}=(\o M_{S,s} \oplus \bb N\alpha \oplus \bb N\beta)/(\alpha+\beta=\delta)
        \]
        for some $\delta\in \o M_{S,s}$. The images of $\alpha,\beta$ in $\o{\ca O}_{X_s,x}$ cut out the two branches of $X_s$ at $x$. We call $\delta$ the \emph{smoothing parameter} of $f$ at $x$.
    \end{itemize}
\end{definition}

\begin{remark}
    By \cite[\S 1]{Kato2000Log-smooth-defo}, log curves are log smooth.
\end{remark}



\subsubsection{Logarithmic and tropical multiplicative groups}

\begin{definition}
    The \emph{log multiplicative group} is the functor
    \begin{align*}
        \bb G_m^{log} \colon \cat{LSch}^{op} & \to \cat{Ab} \\
        T & \mapsto M_T^{gp}.
    \end{align*}
    The \emph{tropical multiplicative group} is the functor
    \begin{align*}
        \bb G_m^{trop} \colon \cat{LSch}^{op} & \to \cat{Ab} \\
        T & \mapsto \o M_T^{gp}.
    \end{align*}
    Both are sheaves in the strict \'etale topology. By \cite[Theorem 3.2]{katoLogStructuresII}), $\bb G_m^{log}$ is even a klf sheaf. If $X$ is a log scheme, we denote by $\bb G_{m,X}^{log}$ the restriction of $\bb G_m^{log}$ to $(\cat{LSch}/X)^{\et}$, and similarly for $\bb G_{m,X}^{trop}$.
\end{definition}

\begin{remark}
    For any log scheme $X$, there is an exact sequence of sheaves of abelian groups on $(\cat{LSch}/X)^{\et}$
    \begin{equation}\label{eqn:fundamental_exact_seq}
        0 \to \bb G_{m,X} \to \bb G_{m,X}^{log} \to \bb G_{m,X}^{trop} \to 0.
    \end{equation}
\end{remark}

\subsection{Tropical and logarithmic Jacobians}\label{section:trop_log_jacobians}

The universal Jacobian $\ca J_g = \Pic^0_{\ca C_g/\ca M_g}$ parametrizes smooth, genus $g$ curves together with an isomorphism class of degree $0$ line bundles on them. It is an abelian variety over $\ca M_g$, but it is not proper since $\ca M_g$ is not. To compactify, one can try to find a natural proper model of $\ca J_g$ over $\o {\ca M}_g$. Many such models have been constructed (see \ref{Prestable} for a sample of references), but none of them is a smooth and proper group over $\o{\ca M}_g$: it can be proven that no such model exists!

On the other hand, over $\ca M_g^{log}$ there is a modular, log smooth, proper group model of $\ca J_g$, the \emph{logarithmic Jacobian} (\emph{log Jacobian} for short) of \cite{molcho_wise_2022}. An intuitive explanation is that when a log curve $X/S$ fails to be smooth over a dense open, the log structure plays the role of an infinitesimal deformation of $\ul X \to \ul S$ to a family of smooth curves. Sections of the log Jacobian play the role of degree $0$ classes of line bundles on the smooth fibres of that deformation. In this section, we collect the properties of the log Jacobian which are relevant to this paper.

\subsubsection{Graphs and tropicalization}

The log Jacobian will be an extension by $\Pic^0$ of a combinatorial object, the \emph{tropical Jacobian} which can be understood in terms of graph homology. We will now discuss tropical Jacobians somewhat informally. We refer to \cite{molcho_wise_2022} for the details and rigorous proofs. Our graphs are allowed multiple edges.

\begin{definition}\label{dualgr}
    A prestable curve $\ul X$ over an algebraically closed field $k$ has a \emph{dual graph} $\Gamma$, whose vertices are the irreducible components of $\ul X$ and whose edges are the nodes. When $\ul X \to \Spec k$ is the underlying scheme of a log curve $X \to (\Spec k,\o P)$ for some sharp fs monoid $\o P$ (see \ref{example:log_point}), we may decorate the edges of the dual graph of $\ul X$ with the smoothing parameters of the corresponding nodes (see \ref{definition:log_curve}). The resulting metric graph, metrized by $\o P$, is called the \emph{dual graph} (or \emph{tropicalization}) of the log curve.
\end{definition}

If $X/S$ is a log curve and $s,t$ are two strict geometric points of $S$ with an \'etale specialization $s \to \Spec \ca O_{S,t}^{\et}$, then the dual graph of $X_s$ is obtained from that of $X_t$ by composing edge lengths with $\o M_{S,t} \to \o M_{S,s}$ and contracting the edges whose length becomes $0$. This allows us to define the \emph{tropicalization} $\ca G_{X/S}$ of $X/S$ (\ref{definition:tropicalization_log_curve}) as a sheaf of combinatorial objects over $S$, which locally are metric graphs metrized by $\o M_S$.

\begin{proposition}\label{proposition:tropicalizing_log_curves}
    Let $X/S$ be a log curve. There exist an \'etale algebraic space $\ca V_{X/S}$ over $\ul S$, a finite and unramified algebraic space $\ca H_{X/S}$ over $\ul S$, a morphism $r_{X/S} \colon \ca H_{X/S} \to \ca V_{X/S}$, an involution $i_{X/S}$ on $\ca H_{X/S}$ and a length map $\ell_{X/S} \colon \ca H_{X/S} \to \o M_S\backslash\{0\}$ such that for all strict geometric points $s \to S$
    \begin{itemize}
        \item $\ca V_{X/S}(s)$ is the set of vertices of $\Gamma_s$.
        \item $\ca H_{X/S}(s)$ is the set of half-edges (or oriented edges) of $\Gamma_s$ and $i_{X/S}$ is the ``opposite orientation" involution.
        \item $r_{X/S}(s)$ takes a half-edge to its source.
        \item $\ell_{X/S}(s)$ is the metric on $\Gamma_s$ by $\o M_{S,s}\backslash\{0\}$.
    \end{itemize}
    The \'etale stalks of $\ca G$ are its geometric fibres, and the restriction maps between them are the edge contractions of dual graphs.
\end{proposition}

\begin{proof}
    See \cite[Section 3]{HMOP}.
\end{proof}

\begin{definition}\label{definition:tropicalization_log_curve}
    Keeping the hypotheses and notations of \ref{proposition:tropicalizing_log_curves} and omitting the subscripts whenever convenient, we call
    \[
    \ca G_{X/S}:=(\ca V_{X/S},\ca H_{X/S},r \colon \ca H_{X/S} \to \ca V_{X/S},i_{X/S},\ell_{X/S})
    \]
    the \emph{tropicalization}, or \emph{sheaf of dual graphs} of $X/S$. We call $\ca V$ its \emph{sheaf of vertices}, $\ca H$ its \emph{sheaf of half-edges}, $i$ its \emph{opposite half-edge involution}, $r$ its \emph{origin map} and $\ell$ its \emph{length map}.
\end{definition}

\begin{definition}\label{definition:sheaf_of_homologies_and_intersection_pairing}
    Let $X/S$ be a log curve. Denote by $\ca Z^{\ca E}_{X/S}$ the sheaf of $(-i_{X/S})$-invariant maps $\ca H_{X/S} \to \bb Z$, which we think of as the sheaf of slopes along the edges of $\ca G_{X/S}$. There is a natural boundary map $\delta_{X/S} \colon \ca Z^{\ca E}_{X/S} \to \bb Z[{\ca V}_{X/S}]$ whose geometric fibres are the usual boundary maps ``endpoint minus origin" in graph homology. The kernel of $\delta_{X/S}$ is the \emph{sheaf of first Betti homologies} $\ca H_{1,X/S}$. It comes with a nondegenerate bilinear \emph{intersection pairing}
    \[
    \ca H_{1,X/S} \times_S \ca H_{1,X/S} \to \bb G_{m,S}^{trop}
    \]
    inherited from the pairing on $\bb Z^{\mathcal H_{X/S}}$ given on characteristic functions of half-edges $h,h'$ by
    \begin{itemize}
        \item $\mathbbm 1_h.\mathbbm 1_{h'}=\ell(h)$ if $h=h'$;
        \item $\mathbbm 1_h.\mathbbm 1_{h'}=-\ell(h)$ if $h=i(h')$;
        \item $\mathbbm 1_h.\mathbbm 1_{h'}=0$ otherwise.
    \end{itemize}
\end{definition}

The tropical Jacobian can then be defined in terms of graph homology as follows.

\begin{definition}\label{definition:tropical_jacobian_bounded_monodromy}
    The \emph{tropical Jacobian} of a log curve $X/S$ is the sheaf
    \[
    \on{TroJac}_{X/S}:=\sheafhom(\ca H_{1,X/S},\bb G_{m,S}^{trop})^\dagger/\ca H_{1,X/S},
    \]
    where $\ca H_{1,X/S}$ maps to $\sheafhom(\ca H_{1,X/S},\bb G_{m,S}^{trop})^\dagger$ via the intersection pairing of \ref{definition:sheaf_of_homologies_and_intersection_pairing}.
    
    Here, the $\dagger$ superscript denotes the subgroup of $\sheafhom(\ca H_{1,X/S},\bb G_{m,S}^{trop})$ consisting of those maps $f \colon \ca H_{1,X/S}(T) \to \o M_T^{gp}$ with \emph{bounded monodromy}. This means that for any cycle $\gamma\in \ca H_{1,X/S}(T)$ there exists a large enough integer $N$ such that
    \[
    -N\ell(\gamma)\leq f(\gamma) \leq N\ell(\gamma)
    \]
    for the partial monoidal order on $\o M_T^{gp}$ induced by $\o M_T$.
\end{definition}

The logarithmic Jacobian will be an extension of the multidegree $0$ subgroup of $\on{Pic}$ by $\on{TroJac}$ (cf. \ref{theorem:representability_and_properties_logjac}).

\begin{example}\label{example:global_sections_trojac_of_elliptic_degen_curve_over_dvr}
    In \ref{example:nonproper_pic_of_elliptic_degeneration}, we have
    \[
    \on{TroJac}_{X/S}(S) = \on{TroJac}_{X'/S}(S) = \bb Z/n\bb Z.
    \]
    Notice that the bounded monodromy condition is automatic on global sections given that $\o M_S=\bb N$ is archimedean. There is a natural map 
    \begin{equation}\label{trewpwf}
        \on{Pic}^{deg=0}_{X'/S}(S) \to \on{TroJac}_{X/S}(S),
    \end{equation}
    taking a line bundle $\ca L$ to $\sum_{i\in \bb Z/n\bb Z} d_i i$ where $d_i$ is the degree of $\ca L$ on the component $D_i$. As expected, piecewise-linear functions are in the kernel of \ref{trewpwf} since for each $i$ the line bundle $\ca O(D_i)$ has degree is $-2$ on $D_i$, $1$ on $D_{i-1}$ and $D_{i+1}$ and $0$ on the other $D_j$.
\end{example}

\subsubsection{Log Jacobians}\label{section:log_lbs}

Our interest in tropicalizations and their sheaf cohomology comes from the fact that they relate to the characteristic monoids of log curves. In essence, this was originally observed by Gross and Siebert \cite[Section 1.4]{GrossSiebert}.

The idea is roughly the following. On a log curve $X/S$, the sheaf $\o M^{gp}$ is locally constant on the smooth locus, so it is determined by its values at nodes and one value per component. In other words, it comes from a sheaf on the dual graph $\ca G_{X/S}$ in some appropriate sense, and this sheaf on $\ca G_{X/S}$ turns out to be piecewise-linear functions. In particular, we have

\begin{lemma}
    Let $X/S$ be a log curve. Recall from \ref{definition:sheaf_of_homologies_and_intersection_pairing} the definitions of the sheaves $\ca Z^{\ca E}$, $\ca H_{1,X/S}$ and their intersection pairings. Then
    \[
    H^1(X,\o M_X^{gp}) = \sheafhom(\ca H_{1,X/S},\bb G_{m,S}^{trop})/\ca Z^{\ca E}(S)
    \]
    with $\ca Z^{\ca E}$ acting via intersection pairing.
\end{lemma}

\begin{proof}
    This follows from \cite[Lemma 2.4.2.4 and Lemma 3.4.7]{molcho_wise_2022}.
\end{proof}

One can then consider the long exact sequence in cohomology attached to the exact sequence of sheaves
\[
0 \to \ca O_X^\times \to M_X^{gp} \to \o M_X^{gp} \to 0
\]
to relate $\on{Pic}$, the cohomology of $M_X^{gp}$, and the cohomology of some natural sheaves on $\ca G_{X/S}$ (such as piecewise-linear functions). This results in the following definition and theorem.

\begin{definition}\label{definition:log_jacobians}
    Let $X/S$ be a log curve. Recall from \ref{definition:tropical_jacobian_bounded_monodromy} the notion of bounded monodromy subgroups, denoted with $\dagger$ superscripts. The \emph{logarithmic Jacobian} $\Logjac_{X/S}$ of $X/S$ is the strict \'etale sheafification of the presheaf $(\cat{LSch}/S)^{op} \to \cat{Ab}$ taking $T$ to the \emph{bounded monodromy subgroup}
    \[
    (H^1(X_T,M_{X_T}^{gp}))^{\mathrm{deg}=0,\dagger} \subset (H^1(X_T,M_{X_T}^{gp}))^{\mathrm{deg}=0},
    \]
    defined as the preimage of
    \[
    \sheafhom(\ca H_{1,X/S},\bb G_{m,S}^{trop})^\dagger(T)/\ca Z^{\ca E}(T)
    \]
    under
    \[
    (H^1(X_T,M_{X_T}^{gp}))^{\mathrm{deg}=0} \to H^1(X_T,\o M_{X_T}^{gp}) = \sheafhom(\ca H_{1,X/S},\bb G_{m,S}^{trop})/\ca Z^{\ca E}(T). 
    \]
    We define a \emph{log line bundle of degree $0$} on $X/S$ to be a global section of $\on{LogJac}_{X/S}$. We will omit the subscripts $X/S$ when possible.
\end{definition}

\begin{remark}
    As the names suggest, the log Jacobian of a log curve $f \colon X \to S$ is the degree $0$ subgroup of a larger subgroup $\on{LogPic}_{X/S}$ of $R^1f_*\bb G_{m,X}^{log}$, whose sections are called \emph{log line bundles}. The degree $0$ part is sufficient for our purposes, so we do not define $\on{LogPic}$.
\end{remark}

\begin{remark}
     We work with sheaves of isomorphism classes rather than stacks of torsors to alleviate notation, not out of necessity.
\end{remark}

Logarithmic Jacobians enjoy the logarithmic analogues of the geometric properties of classical Jacobians of smooth curves, as shown by the following theorem.

\begin{theorem}[Molcho-Wise]\label{theorem:representability_and_properties_logjac}
    Let $X/S$ be a log curve. There is a natural exact sequence of abelian sheaves on $(\cat{LSch}/S)^{\et}$
    \begin{equation}\label{ExactLogPic}
        0 \to \Pic^0_{X/S} \to \Logjac_{X/S} \to \Trojac_{X/S} \to 0
    \end{equation}
    where $\Pic^{0}_{X/S}$ is given the $S$-strict log structure. The sheaf $\Logjac$ is log smooth and proper over $S$, has a log smooth cover by a log scheme, and its diagonal is representable by finite morphisms of log schemes. 
\end{theorem}

\begin{proof}
    The exact sequence is \cite[Theorem 4.14.6]{molcho_wise_2022}. Log smoothness, properness, cover by a log scheme and representability of the diagonal for $\Logjac$ are respectively Theorem 4.13.1, Corollary 4.12.5, Corollary 4.11.4 and Theorem 4.12.1 of \cite{molcho_wise_2022}.
\end{proof}

\begin{example}
    In \ref{example:nonproper_pic_of_elliptic_degeneration}, suppose $R$ is strictly henselian. Let $s \to S$ be the closed point with strict log structure. We have seen in \ref{example:global_sections_trojac_of_elliptic_degen_curve_over_dvr} that $\on{TroJac}_{X/S}(S) \simeq \bb Z/n\bb Z$. Hence, we get from \ref{ExactLogPic} exact sequences of abelian groups
    \begin{align}
        0 \to \on{Pic}^0_{X/S}(S) \to & \Logjac_{X/S}(S) \to \bb Z/n\bb Z\to 0\label{1} \\
        0 \to \bb G_m(s) \to & \on{LogJac}_{X/S}(s) \to \bb Z/n\bb Z \to 0.\label{2}
    \end{align}
    Here, we got \ref{2} from \ref{1} by observing that $\on{Pic}^0_{X_S/S}(s)=\bb G_m(s)$ since
    \begin{itemize}
        \item The normalization $X_s^\nu$ of $X_s$ is $\bb P^1$.
        \item The kernel of the surjective pullback $\on{Pic}(X_s) \to \on{Pic}(X^\nu_s) = \bb Z$ is the isomorphism sheaf
        \[
        \ca Isom(\ca O_{X^\nu_s}|_p,\ca O_{X^\nu_s}|_q) = \bb G_m
        \]
        where $p,q\in X^\nu_s(s)$ are the two preimages of the node.
    \end{itemize}
    With the notations of \ref{example:nonproper_pic_of_elliptic_degeneration}, we have $\on{LogJac}_{X/S}=\on{LogJac}_{X'/S}$, and the kernel of the resulting map $\on{Pic}^{\mathrm{deg}=0}_{X'/S}(S) \to \on{LogJac}_{X/S}(S)$ is the subgroup generated by the line bundles $\ca O(D_i)$.
\end{example}
\subsubsection{Log \'etale descent, lifts to line bundles}\label{section:log_etale_descent}

The definition of the log Jacobian only makes it a sheaf in the strict \'etale topology, but it happens to have much stronger descent properties.

\begin{proposition}\label{proposition:descent_for_logjac}
    Let $X/S$ be a log curve. Then $\Logjac_{X/S}$ satisfies descent along strict fppf covers, root stacks, and those log modifications $T' \to T$ such that $T$ is log flat.
\end{proposition}

\begin{proof}
    This is \cite[Corollary 4.4.14.2]{molcho_wise_2022}.
\end{proof}

The following result formalizes the idea that $\Logjac$ may be thought of as an appropriate quotient of Jacobians.

\begin{proposition}\label{Pic and LogPic}
    Let $X/S$ be a log curve and $\ca L$ a log line bundle of degree $0$ on $X$. There exists a diagram
    \[
    \xymatrix{
    X' \ar[r] & X_{S'} \ar[r]\ar[d] & S' \ar[d] \\
    & X \ar[r] & S \\
    }
    \]
    where $S' \to S$ is a log modification of a root stack, $X' \to X_{S'}$ is a log modification such that $X' \to S'$ remains a log curve, and $\ca L|_{X'}$ lies in the image of $\operatorname{Pic}^{\mathrm{deg}=0}_{X'/S'} \to \operatorname{LogPic}^{\mathrm{deg}=0}_{X'/S'}$.
\end{proposition}

\begin{proof}
    This follows from \cite[Proposition 4.3.2]{molcho_wise_2022}.
\end{proof}

\section{Degenerating torsors}

\subsection{Finite logarithmic group objects}

Following \cite{katoJuin}, we consider several notions of log finite flat groups.

\begin{definition}
    Let $S$ be a log scheme and $S_{klf}$ the category of log schemes over $S$, endowed with the Kummer log flat topology. Denote by $\cat{AbSh}(S_{klf})$ the category\footnote{We ignore set-theoretic issues.} of abelian sheaves on $S_{klf}$. We define the following full subcategories of $\cat{AbSh}(S_{klf})$.
    \begin{enumerate}
        \item $(fin/S)_c$ is the category of finite flat commutative group schemes over $\underline{S}$, equipped with the $S$-strict log structure.
        \item $(fin/S)_f$ is the category of abelian sheaves on $S_{klf}$ whose objects, after a Kummer log flat base change over $S$, belong to $(fin/S)_c$.
    \end{enumerate}
\end{definition}

By \cite[\S 1.5]{katoJuin}, the \emph{Cartier dual} endofunctor on $\cat{AbSh}(S_{klf})$
\begin{equation*}
    G \mapsto G^D:=\sheafhom_S(G,\bb G_m)
\end{equation*}
preserves $(fin/S)_f$. \\

\begin{example} We give two examples of objects in $(fin/S)_f$:\\
\begin{itemize} 
    \item (\cite[\S 1.8]{katoJuin})  Let $S$ be a log scheme and $n \geq 1$. Let 
    $$1 \to \mu_n \to \mathbb{G}_m^{log} \xrightarrow{\times n} \mathbb{G}_m^{log} \to 1$$
    be the Kummer exact sequence in the klf site. Let $\delta_n$ be the connecting map $H^0(S, \mathbb{G}_m^{log}) \to H^1(S,\mu_n)$. Let $a \in \Gamma(S,M_S^{gp})$. The section $\delta_n(a)$ yields an extension
    $$ 1 \to \mu_n \to G \to \mathbb{Z}/n\mathbb{Z} \to 1,$$ 
and the resulting $G$ is an object in $(fin/S)_f$. Indeed, it is a $\mu_n$-klf torsor over $\mathbb{Z}/n\mathbb{Z}$, hence representable by a finite flat group scheme of Kummer type (cf. \cite[Theorem 9.1]{katoLogStructuresII}). But if there exists a point $x \in X$ such that the image of $a$ in $\overline{M}_{X,\overline{x}}^{gp}$ is not an $n$-th power, $G$ does not belong to $(fin/S)_c$.
    \item The $n$-torsion of the log Jacobian of any log curve $X/S$ is an object of $(fin/S)_f$. The smallest klf cover $\pi \colon S' \to S$ such that $\pi^*\on{LogPic}_{X/S}[n]$ is an object of $(fin/S')_c$ is a root stack, explicitly described in \cite[Section 4.2]{Holmes2022Logarithmic-mod}. In fact, by \cite[Proposition 4.5]{katoJuin} the $n$-torsion of any log abelian variety in the sense of \cite{IIKajiwara2008Logarithmic-a} is an object of $(fin/S)_f$, and log Jacobians are log abelian varieties by \cite[Theorem 4.15.7]{molcho_wise_2022}.
\end{itemize}
   
\end{example}

\subsection{Kummer log flat torsors}
\begin{definition}
    Let $S$ be a log scheme, $X$ an $S$-log scheme and $G$ an object of $(fin/S)_f$. A klf $G$-torsor (or \emph{log $G$-torsor}) $T$ over $X$ is a $G$-principal homogeneous space in $X_{klf}$. By definition of $(fin/S)_f$, there exists a Kummer log étale cover $S' \to S$ such that $G_{S'} \to S$ is finite, flat and strict. By \cite[Theorem 9.1]{katoLogStructuresII}, the base change $T\times_S S'$ is then representable by a log scheme which is Kummer log flat over the base, with a finite underlying scheme over the base.\\ 
\end{definition}
\begin{example}
In the setting of \ref{Klflat_morphism}, suppose that $A$ contains an $n$-th root of unity. The morphism $f \colon \Spec B \to \Spec A$ is a $\bb Z/n\bb Z$-Kummer log flat (in fact, Kummer log \'etale) torsor. Since $f$ is ramified, it is not a fppf torsor.
\end{example}

\subsection{Extending torsors into log torsors}
Assume that $S$ is log regular and let $U\subset S$ be the dense open where the log structure is trivial. Let $X/S$ be a log curve (then $X_U/U$ is smooth), $G$ an element of $(fin/S)_f$ and $Y \to X_U$ a fppf $G_U$-torsor. There does not always exist an fppf $G$-torsor over $X$ extending $Y \to X_U$ (even when $S$ is the spectrum of a DVR, see e.g. \cite[\S 4.1]{Sara}). Asking for an extension of $Y \to X_U$ to a \emph{log} torsor over $X$ is more reasonable. Such extensions are always unique by \ref{Uniqueness}. Our goal in this section is to prove that the log Jacobian of $X/S$ parametrizes log $G$-torsors (\ref{mainrslt}) and deduce an existence result for log extensions of torsors (\ref{maincor}).

\subsubsection{Uniqueness of extensions}
\begin{lemma}\label{uniq}
Let $g: T \to X$ be a finite morphism of schemes with $X$ normal and integral. Any rational section of $g$ is defined everywhere.
\end{lemma}
\begin{proof}
    This is a special case of \cite[Corollary 6.1.15]{Groth}.
\end{proof}

\begin{lemma}\label{lemma:restriction_group_is_bijective}
    Let $Y$ be a log regular log scheme, $Y_U\subset Y$ a strict, scheme-theoretically dense open and $G_Y$ a finite, flat, strict and commutative $Y$-group. Then, the restriction map
    \begin{equation*}
        G_Y(Y) \to G_Y(Y_U)
    \end{equation*}
    is bijective.
\end{lemma}

\begin{proof}
    Since $G_Y \to Y$ is strict, it suffices to check that the scheme-theoretic restriction
    \begin{equation}\label{polarbear}
        \ul G_Y(\ul Y) \to \ul G_Y(\ul Y_U)
    \end{equation}
    is bijective. Surjectivity is \ref{uniq} since $\ul Y$ is normal. The map $G_Y \to Y$ is finite, hence separated, and $U$ is scheme-theoretically dense in $Y$ so \ref{polarbear} is injective as well.
\end{proof}

\begin{corollary}\label{Uniqueness}
   Let $S$ be a log regular log scheme and $X/S$ a log curve, smooth over the dense open $U\subset S$ of triviality of $M_S$. Let $G/S$ be a strict finite flat commutative group scheme. The restriction map
   \begin{equation}\label{eqn:jiehps}
       H^1_{klf}(X,G) \to H^1_{fppf}(X_U,G_U)
   \end{equation}
   is injective.
\end{corollary}

\begin{proof}
Let $T$ be a $G_X$-torsor over $X$ and let $X' \to X$ be a klf covering which trivializes $T$. Let $X'':=X'\times_X X'$ and $X''':=X'\times_X X' \times_X X'$. Since $X',X'',X'''$ are log smooth over $S$, they are log regular. Since $T$ is trivial over all of them, by \ref{lemma:restriction_group_is_bijective} the restriction maps
\begin{align*}
    T_{X'}(X') & \to T_{X'}(X'\times_S U) \\
    T_{X''}(X'') & \to T_{X''}(X''\times_S U) \\
    T_{X'''}(X''') & \to T_{X'''}(X'''\times_S U)
\end{align*}
are bijective. Thus, by descent, 
\begin{equation}\label{penguin}
   T(X) \to T(X_U)
\end{equation}
is bijective as well. In particular, $T \to X$ has a section if and only if $T_U \to X_U$ does, i.e. \ref{eqn:jiehps} is injective as claimed.
\end{proof}

\subsubsection{Some preliminary lemmas}
Log curves satisfy the following logarithmic analogue of cohomological flatness in dimension $0$:
\begin{lemma}\label{f_*M_X=M_S} Let $f\colon X \to S$ be a log curve. The natural map
\begin{equation}\label{eqn:log_cohomological_flatness}
    \mathbb{G}_{m,S}^{log} \to f_*\mathbb{G}_{m,X}^{log}
\end{equation}
is an isomorphism.
\end{lemma}

\begin{proof}
    This follows from \cite[Lemma 4.6.1]{molcho_wise_2022}. We give the idea here. The formation of \ref{eqn:log_cohomological_flatness} commutes with base change, so it suffices to check that it is an isomorphism on global sections. Working \'etale-locally and by a limit argument, we may assume $S$ is local and strictly henselian. The diagram of abelian sheaves on $S_{\et}$
    \[
    \xymatrix{
    0 \ar[r] & \ca O_S^\times\ar[r]\ar[d] & M_S^{gp}\ar[r]\ar[d] & \o M_S^{gp}\ar[r]\ar[d] & 0 \ar[d] \\
    0 \ar[r] & f_*\ca O_X^\times\ar[r] & f_* M_X^{gp}\ar[r] & f_* \o M_X^{gp}\ar[r] & \on{Pic}_{X/S}.
    }
    \]
    has exact rows. By \cite[\href{https://stacks.math.columbia.edu/tag/0GKA}{Tag 0GKA}]{stacks-project}, the first vertical arrow is an isomorphism. Hence, it suffices to check that the kernel of $f_* \o M_X^{gp}(S) \to \on{Pic}_{X/S}(S)$, which contains $\o M_S^{gp}(S)$, is exactly $\o M_{S}^{gp}(S)$. Let $s$ be the closed point of $S$ and $\mathfrak X_s$ be the tropicalization of $X_s/s$, with edge lengths in $\o M_S(S)$. The group $f_* \o M_X^{gp}(S)=f_* \o M_X^{gp}(s)$ is naturally identified with the group $\on{PL}(\mathfrak X_s)$ of $\o M_S(S)$-valued piecewise-linear functions on $\mathfrak X_s$ (with integer slopes along edges). This identifies $\o M_{S}^{gp}(S) \hra f_*\o M_X^{gp}(S)$ with the subgroup of $\on{PL}(\mathfrak X_s)$ consisting of constant functions. Thus, we have reduced to showing that any function $\phi$ in the kernel of
    \begin{align*}
        \on{PL}(\mathfrak X_s) & \to \on{Pic}_{X/S}(S) \\
        \psi & \mapsto \ca O(\psi)
    \end{align*}
    is constant. Let $v$ be a vertex of $\mathfrak X_s$ at which $\phi$ attains a minimal value $m$ for the partial monoidal order on $\o M^{gp}$ induced by $\o M$. By the minimality of $m$, for each half-edge $e$ exiting $v$, the slope of $\phi$ along $e$ is nonnegative. These slopes sum to the degree of the line bundle $\ca O(\phi)$ on the irreducible component of $\ul X_s$ corresponding to $v$, so they are all zero. Hence, $\phi$ also takes the value $m$ at every vertex neighbouring $v$. Applying this to each vertex in $\phi^{-1}(\{m\})$, and since $\mathfrak X_s$ is connected, we deduce that $\phi$ is constant as claimed.
\end{proof}

Given a log scheme $Y$, we denote by $\cat{Ab}_{klf}/Y$ the category of sheaves of abelian groups on the big klf site $(\cat{LSch}/Y)^{klf}$. We denote by $\sheafhom_{klf},\sheafext_{klf}$ the $\operatorname{Hom}$ and $\operatorname{Ext}$ operators on $\cat{Ab}_{klf}$. We omit the subscripts when we can safely do so.

\begin{lemma}\label{Hom and Ext}
 Let $X/S$ be a log curve and $G$ an object in $(fin/S)_f$. There are isomorphisms:
\begin{align}\label{Hom}
\sheafhom(G_{X},\mathbb{G}_{m,X_{}}^{log}) \simeq \sheafhom(G_{X},\mathbb{G}_{m,X_{}}),
\end{align}
\begin{align}\label{Ext}
\sheafext^1(G_{X},\mathbb{G}_{m,X_{}}^{log}) \simeq \sheafext^1(G_{X_{}},\mathbb{G}_{m,X_{}}) \simeq \{0\}.
\end{align}
\end{lemma}

\begin{proof}
 We have an exact sequence in $\cat{Ab}_{klf}/X$
\begin{equation}\label{eqn:fundamental_exact_seq_klf_site}
    0 \to  \mathbb{G}_{m,X_{}} \to \mathbb{G}_{m,X_{}}^{log} \to \mathbb{G}_{m,X_{}}^{trop,klf} \to 0,
\end{equation}
where $\mathbb{G}_{m,X_{}}^{trop,klf}$ denotes the klf sheafification of $\bb G_{m,X}^{trop}$.
This yields an exact sequence:
\begin{align*}
     0 &\to \sheafhom(G_{X_{}},\mathbb{G}_{m,X_{}}) \to \sheafhom(G_{X_{}},\mathbb{G}_{m,X_{}}^{log}) \to \sheafhom(G_{X_{}},\mathbb{G}_{m,X_{}}^{trop,klf})\\
     &\to \sheafext^1(G_{X_{}},\mathbb{G}_{m,X_{}}) \to \sheafext^1(G_{X},\mathbb{G}_{m,X_{}}^{log}) \to  \sheafext^1(G_{X_{}},\mathbb{G}_{m,X_{}}^{trop,klf}).
\end{align*}
Since $\mathbb{G}_{m,X_{}}^{trop,klf}$ has no torsion (cf. \cite[Lemma 2.3.1]{Gill}), the isomorphism \ref{Hom} follows.\\

The second isomorphism of \ref{Ext} follows from \cite[Theorem 4.1]{Gill'}. This, together with the exactness of \ref{eqn:fundamental_exact_seq_klf_site}, gives an injection $\sheafext^1(G_{X},\mathbb{G}_{m,X_{}}^{log}) \hookrightarrow \sheafext^1(G_{X},\mathbb{G}_{m,X_{}}^{trop, klf})$. Therefore, it suffices to show that $\sheafext^1(G_{X},\mathbb{G}_{m,X_{}}^{trop, klf})=0$. For any nonzero integer $n$, multiplication by $n$ is an automorphism of $\mathbb{G}_{m,X_{}}^{trop, klf}$ (cf.\cite[Lemma 2.3.1]{Gill}), hence of $\sheafext^1(G_{X},\mathbb{G}_{m,X_{}}^{trop, klf})$ by functoriality. By \cite[\S 2]{katoJuin}, there exists a nonzero integer $n_0$ which kills $G_{X}$. Thus, the multiplication by $n_0$ automorphism of $\sheafext^1(G_{X},\mathbb{G}_{m,X_{}}^{trop, klf})$ is zero and $\sheafext^1(G_{X},\mathbb{G}_{m,X_{}}^{trop, klf})$ is trivial as claimed.
\end{proof}

\begin{lemma}\label{raynaudlog}
 Let $f: X \to S $ be a log curve and let $G$ be an object in $(fin/S)_f$. We have a canonical isomorphism:
 $$R^1_{klf}f_*(G_{X_{}}^D) \iso \sheafhom(G, R^1_{klf}f_*\mathbb{G}_{m,X_{}}^{log}). $$
 \end{lemma}
 
 \begin{proof} We follow the proof of \cite[Proposition 6.2.1]{Raynaud1970Specialisation-}. Consider the functor
\begin{align*}
    H \colon \cat{Ab}_{klf}(X) & \to \cat{Ab}_{klf}(S) \\
    F & \mapsto f_*\sheafhom(f^*G,F) = \sheafhom(G,f_* F).
\end{align*}
Since we have written $H$ as a composite of functors in two different ways, we get two convergent Grothendieck spectral sequences (cf.\cite{GrotSpectral}):
\begin{center}
\begin{align*}
R^{p+q}H(F) \Rightarrow & R^qf_*\sheafext^p(f^*G,F)\\
R^{p+q}H(F) \Rightarrow & \sheafext^q(G,R^pf_*F).
\end{align*}
\end{center}
These spectral sequences yield exact sequences of low-degree terms:
\begin{equation}
    0 \to R^1f_*(\sheafhom(G_{X_{}},F)) \to R^1H(F) \to f_*\sheafext^1(G_{X_{}},F)\label{eqn:short_exact_lowdeg_1}
\end{equation}
and
\begin{equation}
    0 \to \sheafext^1(G_{},f_*F)  \to R^1H(F) \to \sheafhom(G_{},R^1f_*F) \to \sheafext^2(G_{},f_*F) \to R^2H(F).\label{eqn:short_exact_lowdeg_2}
\end{equation}

By \ref{Hom and Ext}, the rightmost term of \ref{eqn:short_exact_lowdeg_1} vanishes if $F=\bb G_{m,X}^{log}$, so
\[
R^1f_*(\sheafhom(G_{X_{}},\mathbb{G}_{m,X_{}}^{log})) = R^1H(\mathbb{G}_{m,X_{}}^{log}).
\]
Plugging this into \ref{eqn:short_exact_lowdeg_2}, we get an exact sequence
\[
0 \to R^1_{klf}f_*(G_{X_{}}^D)  \to \sheafhom(G_{},R^1_{klf}f_*\mathbb{G}_{m,X_{}}^{log}) \to \sheafext^2_{klf}(G_{},f_*\mathbb{G}_{m,X_{}}^{log})   \xrightarrow{\gamma} R^2H(\mathbb{G}_{m,X_{}}^{log}).
\]
It remains to see that $\gamma$ is injective to conclude. Working locally on $S$, we may assume that $X_{} \to S$ has a section $s$ through the strict (equivalently, smooth) locus. Denote by $s^{-1}$ the pullback functor $\cat{Ab}_{klf}/X \to \cat{Ab}_{klf}/S$ along $s$. For any morphism of log schemes $T \to S$ we have $s^{-1}F(T \to S)=F(T \to S \xrightarrow{s} X)$ so $s^{-1}$ commutes with kernels, i.e. is left exact. 
Since $s^{-1}$ is a left adjoint, it is also right exact and hence exact  (see also  \cite[\href{https://stacks.math.columbia.edu/tag/00XR}{Tag 00XR}]{stacks-project}). Therefore, precomposition with $s^{-1}$ preserves left exactness and commutes with $R^i$ for all $i>0$. In particular, the functor
\begin{align*}
    H' \colon \cat{Ab}_{klf}(X) & \to \cat{Ab}_{klf}(S) \\
    F & \mapsto \sheafhom(G,s^{-1}F)
\end{align*}
is left exact and $R^2H' = F \mapsto \sheafext^2(G,s^{-1}F)$. The natural transformation $f_* \to s^{-1}$ yields a morphism $R^2H \to R^2H'$. The composite
\begin{equation}\label{lion}
    \sheafext^2(G_{},f_{*}\bb G_{m,X}^{log}) \xrightarrow{\gamma} R^2H(\mathbb{G}_{m,X_{}}^{log}) \to R^2H'(\bb G_{m,X}^{log}) = \sheafext^2(G_{},s^{-1}\bb G_{m,X}^{log})
\end{equation}
is an isomorphism since $f_*\mathbb{G}_{m,X_{}}^{log}=\mathbb{G}_{m,S}^{log}$ by \ref{f_*M_X=M_S} and $s^{-1}\bb G_{m,X}^{log} = \bb G_{m,S}^{log}$ by the strictness of $s$ combined with the equality $s^{-1}\bb G_{m,X} = \bb G_{m,S}$. Therefore, $\gamma$ is injective as claimed.
\end{proof}

\begin{remark}
    The previous lemma remains true by an almost identical proof if one replaces $\mathbb{G}_m^{log}$ by $\mathbb{G}_m$. In particular, for a finite flat commutative group scheme $G$,
   \[
   \sheafhom(G, R^1_{klf}f_*\mathbb{G}_{m,X_{}}^{})\simeq \sheafhom(G, R^1_{klf}f_*\mathbb{G}_{m,X_{}}^{log}).
   \]
 We choose to use $\bb G_m^{log}$ and not $\bb G_m$ because some natural, geometrically meaningful subgroup of $R^1_{klf}f_*\mathbb{G}_{m,X_{}}^{log}$ is represented by the log Jacobian (which is a priori defined as parametrizing some \emph{strict \'etale} $\bb G_m^{log}$-torsors), while the comparison between $R^1_{klf}f_*\mathbb{G}_{m,X_{}}$ and $R^1_{\acute etale}f_*\mathbb{G}_{m,X_{}}$ is less straightforward.
\end{remark}

\begin{proposition}\label{pointedlog}
Let $f: X \to S$ be a log curve and let $G$ an element of $(fin/S)_f$. We have a canonical isomorphism:
 $$H^1_{klf}(X,G)/H^1_{klf}(S,G) \xrightarrow{\simeq} R^1_{klf}f_*G_{X}(S). $$
\end{proposition} 
\begin{proof}

Working locally on $S$ we may assume there is a section $S \to X$ through the smooth (and strict) locus of $X/S$. Let us write the Leray sequence associated to $f$ and $G_{X}$:

\begin{small}
$0 \to H^1_{klf}(S, f_*G_{X}) \to H^1_{klf}(X,G) \to H^0(S,R^1_{klf}f_*G_{X}) \to H^2_{klf}(S, f_*G_{X}) \xrightarrow{\delta} H^2_{klf}(X,G).$
\end{small}
\newline
 On the other hand, we have
 \begin{align*}
 f_*G_{X}=f_* \sheafhom(G_{X}^D,\mathbb{G}_{m,X}) & = f_*\sheafhom(f^*G^D,\mathbb{G}_{m,X})\\
 &\simeq \sheafhom(G^D,f_*\mathbb{G}_{m,X})\\
 &=\sheafhom(G^D,\mathbb{G}_{m,S})=G
 \end{align*}
 where the second to last equality follows from the fact that $f_*\mathcal{O}_{X} \simeq \mathcal{O}_S$. Hence, $f_*G_{X} \simeq G$.\\
In addition, $\delta$ is injective since $f$ has a section. Plugging this into the Leray sequence yields the desired exact sequence
\begin{equation*}
0 \to H^1_{klf}(S, G) \to H^1_{klf}(X,G) \to H^0(S,R^1_{klf}f_*G_{X}) \to 0.
\end{equation*}
 \end{proof}

\subsubsection{Existence of extensions}

\begin{theorem}\label{mainrslt}
Let $X/S $ be a log curve and let $G$ be a element of $(fin/S)_f$. There is a canonical isomorphism
 \[
 H^1_{klf}(X_{},G_{})/H^1_{klf}(S,G) \simeq   {\Hom}(G^D, \mathrm{LogJac}_{X/S}).
 \]
\end{theorem}

\begin{proof} 
    By \ref{raynaudlog} and \ref{pointedlog}, it suffices to show that the natural map
    \[
    \sheafhom(G^D, \operatorname{LogJac_{X/S}}) \to \sheafhom(G^D, R^1_{klf}f_*\mathbb{G}_{m,X_{}}^{log})
    \]
    is an isomorphism. Since $G^D$ is torsion, this reduces to proving that the quotient $$R^1_{klf}f_*\mathbb{G}_{m,X_{}}^{log}/\operatorname{LogJac_{X/S}}$$ is torsion-free. Quotienting the numerator and denominator by $\on{Pic}^0$, we get
    \begin{align*}
        R^1_{klf}f_*\mathbb{G}_{m,X_{}}^{log} / \operatorname{LogJac}_{X/S} 
        & =R^1_{klf}f_*\mathbb{G}_{m,X_{}}^{trop}/\operatorname{TroJac}_{X/S} \\
        & = \left(\sheafhom(\ca H_{1,X/S},\bb G_m^{trop})/\ca H_{1,X/S}\right)/\left(\sheafhom(\ca H_{1,X/S},\bb G_m^{trop})^\dagger/\ca H_{1,X/S}\right) \\
        & = \sheafhom(\ca H_{1,X/S},\bb G_m^{trop})/\sheafhom(\ca H_{1,X/S},\bb G_m^{trop})^\dagger,
    \end{align*}
    which is torsion-free since the stalks of $\sheafhom(\ca H_{1,X/S},\bb G_m^{trop})^\dagger$ are saturated sublattices of the stalks of $\sheafhom(\ca H_{1,X/S},\bb G_m^{trop})$.
\end{proof}

\begin{corollary}\label{maincor}
    Let $X/S$ be a log curve with $S$ log regular. Let $G$ be an element of $(fin/S)_f$ such that $G^D$ is Kummer log \'etale over $S$. Then, any relative fppf $G_U$-torsor over $X_U$ extends uniquely to a relative klf $G$-torsor over $X$, where $G$ is the Cartier dual of $G^{D}/S$.
\end{corollary}

\begin{proof}
By \cite[Theorem 6.11]{HMOP}, $\mathrm{LogJac}_{X/S}$ is the Néron model (in the log smooth topology) of the Jacobian $J$ of $X_U/U$. In particular, $$\mathrm{LogJac}_{X/S}(G^D) \to \mathrm{LogJac}_{X/S}(G_U^D)=J(G_U^D)$$ is bijective and we conclude using \ref{mainrslt}.
\end{proof}

\subsection{Examples and non-examples}\label{section:examples}

\subsubsection{Torsors under $\mu_n$, $\bb Z/n\bb Z$ and $\on{LogPic}[n]$}

 The first proposition follows immediately from  \ref{maincor}. 
\begin{proposition}[Some groups with \'etale Cartier duals]\label{ext_torsors}
   Let $S$ be log regular and let $X/S$ be a log curve. Let $U\subset S$ be the dense open where $\ca O_S^\times \hra M_S$ is an isomorphism. In particular, $X_U/U$ is strict and smooth. Then for any log finite and flat commutative group $G/S$ whose Cartier dual $G^D$ is log \'etale, $G_U$-torsors on $X_U/U$ extend uniquely to $X/S$. Examples include
   \begin{itemize}
       \item $G=\mu_n$ for any $n$ (so that $G^D=\bb Z/n\bb Z$).
       \item Any $G/S$ whose order is invertible on $S$ (or equivalently which is killed by an integer invertible on $S$), e.g. $\bb Z/n\bb Z$ or $\on{LogPic}[n]$ for $n\in \ca O_S^\times(S)$.
   \end{itemize}

\end{proposition}

On the other hand, when $G^D/S$ is not log \'etale, $G$-torsors on curves may not extend uniquely over log regular bases. Still, we can use \ref{mainrslt} to understand their structure. We give some examples below.

\begin{proposition}[Non-extension of $\mathbb{Z}/p\mathbb{Z}$-torsors.]\label{non-ext}
      In \ref{example:nonproper_pic_of_elliptic_degeneration}, suppose the discrete valuation ring $R$ has mixed characteristic $(0,p)$ and contains a primitive $p$-th root of unity and a primitive $n$-th root of unity. We have
    \begin{equation}\label{monkey}
          \on{LogPic}_{X/S}[n]=\mu_n \times \bb Z/n\bb Z.
      \end{equation}
    as strict $S$-group schemes. In particular,
    \begin{itemize}
        \item If $n$ is prime to $p$, then log torsors under $\mathrm{LogPic}_{X/S}[n]$ over $X/S$ are equivalent to fppf torsors under $\mathrm{Pic}_{X_{\eta}/\eta}[n]=\mathrm{Pic}^0_{X_{\eta}/\eta}[n]$ on the generic fiber $X_{\eta}/\eta$.
        \item If $n=p$, then one can construct a family of fppf $\mathbb{Z}/p\mathbb{Z}$-torsors over $X_{\eta}/\eta$ which do not extend to log $\mathbb{Z}/p\mathbb{Z}$-torsors over $X/S$.
        
\end{itemize}
    \end{proposition}    
        
     \begin{proof} The first statement follows from \ref{ext_torsors}.
         We prove the second statement. Any morphism $\mu_{p,\eta} \simeq \bb Z/p\bb Z \to \on{LogPic}_{X_\eta/\eta}[p]$ whose projection onto the second factor $\bb Z/p\bb Z$ of \ref{monkey} is nonzero has no chance of extending to a global map $\mu_{p,S} \to \on{LogPic}_{X/S}[p]$ since $\mu_{p,S}$ is connected. Hence, the corresponding $\bb Z/p\bb Z$-torsors on $X_\eta/\eta$ do not extend to $X/S$.

     \end{proof}

\begin{proposition} \label{LogJac[n]} (The universal $\on{LogPic}[n]$-torsor)
$\on{LogPic}[n]$-torsors on a log curve can be naturally identified with endomorphisms of $\on{LogPic}[n]$. In particular, the inclusion $\mathrm{LogPic}_{X/S}[n] \to \mathrm{LogPic}_{X/S}$ gives rise to a \emph{universal $\mathrm{LogPic}_{X/S}[n]$-torsor} $\ca T$. Furthermore, for any log finite flat $S$-group scheme $G$ killed by $n$, any $G$-torsor $T$ on $X$ corresponds to a map $G^D \to \mathrm{LogPic}_{X/S}[n]$, whose Cartier dual is a map
\begin{equation}\label{rabbit}
    \mathrm{LogPic}_{X/S}[n] \to G.
\end{equation}
from which one obtains $T$ from $\ca T$ by change of group structure.
\end{proposition}

\begin{proof}
    
For any abelian variety $A/S$, the Cartier dual of the $n$-torsion $A[n]$ is the $n$-torsion of the dual abelian variety $A^\vee/S$. Jacobians are canonically polarized by their theta divisors, so they are self-dual as abelian varieties. In particular, the $n$-torsion of a Jacobian is canonically its own Cartier dual. By \cite{MolchoUlirschWiseLogDelignePairing}, this self-duality extends to a natural self-Cartier duality of $\on{LogPic}[n]$. Therefore, by \ref{mainrslt}, $\on{LogPic}[n]$-torsors on a log curve are naturally identified with endomorphisms of $\on{LogPic}[n]$.\footnote{In general, if $S$ is a log regular base and $U$ the open of triviality of $M_S$, such endomorphisms over $U$ may not extend to endomorphisms over $S$ (so the corresponding torsors do not extend either). This can be seen, for example, for $n=p$ over the curve of \ref{non-ext} since the $p$-torsion of its log Picard group contains a copy of $\mu_{p,S}$.}

On the other hand, there is always a \emph{universal $\mathrm{LogPic}_{X/S}[n]$-torsor} $\ca T$, given by the inclusion $\mathrm{LogPic}_{X/S}[n] \to \mathrm{LogPic}_{X/S}$. The universal property of $\ca T$ is the following. For any log finite flat $S$-group scheme $G$ killed by $n$, the Cartier dual $G^D$ is also killed by $n$. Hence, any $G$-torsor $T$ on $X$ corresponds to a map $G^D \to \mathrm{LogPic}_{X/S}[n]$, whose Cartier dual is a map $\mathrm{LogPic}_{X/S}[n] \to G$.
One now obtains $T$ from $\ca T$ by change of structure group using \ref{rabbit}.
\end{proof}

\begin{example}\label{example:univ_torsor_on_nodal_cubic}[Universal torsors on nodal cubics]
    Let $X/S$, $n$ and $X'$ be as in \ref{example:nonproper_pic_of_elliptic_degeneration}. Give $S$ the divisorial log structure and $X,X'$ their structures of log curves over $S$, so that $f \colon X' \to X$ is a log blowup. We claim that the universal $\on{LogPic}_{X/S}[n]$-torsor on $X/S$ is a klf cover $\ca T \colon X' \to X$ ($\ca T$ is not equal to $f$). We will now construct this klf cover $\ca T$, and only sketch the proof that it is the universal torsor. Let $0_{X/S} \colon S \to X$ be a smooth section through $D_0$. We may see the generic fibre $X_\eta/\eta=X'_\eta/\eta$ as an elliptic curve by declaring the origin to be (the restriction of) $0_{X/S}$. The abelian varieties $X_\eta/\eta$ and $\on{Pic}^0_{X_\eta/\eta}$ are dual, and canonically isomorphic via the Abel-Jacobi isomorphism
\begin{align*}
   X_\eta & \iso \on{Pic}^0_{X_\eta/\eta} \\
   x & \mapsto \ca O([x-0_{X/S}]).
\end{align*}
There is an analogous ``log Abel-Jacobi map based at $0_{X/S}$" $X \to \on{LogPic}^0_{X/S}$, such that for any map $T \to S$ and any log blowup $Y \to X_T$ which remains a log curve over $S$, the square
\[
\begin{tikzcd}
    (Y/T)^{smooth} \arrow[r] \arrow[d] & \on{Pic}_{Y/T}^{\deg=0} \arrow[d] \\
    X \arrow[r] & \on{LogPic}^0_{X/S}
\end{tikzcd}
\]
commutes. Here, the rightmost vertical arrow is the composite
\[
\on{Pic}_{Y/T}^{\deg=0} \to \on{LogPic}^0_{Y/T} = \on{LogPic}^0_{X_T/T} \to \on{LogPic}^0_{X/S}
\]
and the top horizontal arrow is the classical Abel-Jacobi embedding $y \mapsto \ca O([y-0_{X/S}])$. Since any $T$-point of $X$ factors through $(Y/T)^{smooth}$ for some $Y$, this property uniquely characterizes the log Abel-Jacobi map.

The multiplication by $n$ endomorphism of $\on{LogPic}^0_{X/S}$ fits into a cartesian square
\begin{equation}\label{commdiag}
\begin{tikzcd}
    X' \arrow[r,"\times n"] \ar[dr, phantom, "\square"]\arrow[d] & X \arrow[d] \\
    \on{LogPic}^0_{X/S} \arrow[r,"\times n"] & \on{LogPic}^0_{X/S}.
\end{tikzcd}
\end{equation}
We emphasize that the map $\times n \colon X' \to X$ in \ref{commdiag} is \emph{not} the log blowup by which we defined $X'$: it is not a log blowup at all, but a klf cover. Explicitly, recall that the local equation for $X$ at the node of the special fibre $X_s$ is of the form $uv=\pi^n$ where $u,v$ are parameters for the branches of $X_s$ and $\pi$ is a uniformizer for $S$. The map $\times_n$ factors as
\[
X' \to X^{n-root} \to X,
\]
where $X' \to X^{n-root}$ is a strict fppf $\on{LogPic}_{X/S}[n]$-torsor and $X^{n-root} \to X$ is the root stack of index $n$ obtained from $X$ by adjoining $n$-th roots $\sqrt[n]{u},\sqrt[n]{v}$ to $u,v$, imposing the relation $\sqrt[n]{u}\sqrt[n]{v}=\pi$ and modding out the action of $\mu_n$ with weights $(1,-1)$ on $(\sqrt[n]{u},\sqrt[n]{v})$.

In particular, the map $\times n \colon X' \to X$ is a $\on{LogPic}_{X/S}[n]$-torsor in the klf topology, which we denote by $\ca T$. We claim that $\ca T$ is the universal torsor of \ref{LogJac[n]}. \\

We will only give a rough sketch of the proof of this claim. First, we must discuss Weil pairings and their degenerations. The classical Weil pairing
\[
X_\eta[n] \times X_\eta[n] \to \mu_n \hra \bb G_m
\]
over $\eta$ extends uniquely to a non-degenerate log Weil pairing over $S$
\[
\on{Weil} \colon \on{LogPic}_{X/S}[n] \times \on{LogPic}_{X/S}[n] \to \mu_n \hra \bb G_m^{log}.
\]
Explicitly, pick a generator $\gamma$ for the first homology of the dual graph $\Gamma$ of $X_s/s$ (which is topologically a circle). This induces isomorphisms
\begin{align*}
    \on{Pic}^0_{X_s/s} = \on{Hom}(H_1(\Gamma),\bb G_{m,s}) & \iso \bb G_{m,s} \\
    f & \mapsto f(\gamma) \\
    \on{Pic}_{X_s/s}[n] = \on{Hom}(H_1(\Gamma),\mu_{n,s}) & \iso \mu_{n,s} \\
    f & \mapsto f(\gamma) \\
    \on{TroPic}_{X_s/s}[n] = \on{H_1(\Gamma)}/nH_1(\Gamma) & \iso \bb Z/n\bb Z \\
    \gamma & \mapsto 1
\end{align*}
and over $s$, the Weil pairing is the map
\begin{align}\label{formula_weil_pairing}
    \on{LogPic}_{X_s/s}[n] \times \on{LogPic}_{X_s/s}[n] & \to \on{Pic}_{X_s/s}[n] \iso \mu_n \\
    (a,b) & \mapsto a^{\on{Trop}(b)}b^{-\on{Trop}(a)},
\end{align}
where $\on{Trop}$ is the composite $\on{LogPic}_{X/S}[n] \to \on{TroPic}_{X/S}[n] \iso \bb Z/n\bb Z$.\footnote{Choosing the other generator $-\gamma$ for $H_1(\Gamma)$ would multiply the isomorphisms $\on{TroPic}_{X/S}[n] \iso \bb Z/n\bb Z$ and $\mu_n \iso \on{Pic}_{X_s/s}[n]$ by $-1$, resulting in the same log Weil pairing.}

The log Weil pairing can be geometrically interpreted as follows. The Poincar\'e bundle on $X_\eta \times_\eta \on{Pic}^0_{X_\eta/\eta}$ corresponds via Abel-Jacobi to a $B\bb G_m$-valued bilinear pairing on $\on{Pic}^0_{X_\eta/\eta}$. This extends uniquely to an analogous bilinear pairing of stacks over $S$\footnote{Unless specified otherwise, all our classifying stacks are taken in the klf topology. For the sheaf $\bb G_m^{log}$, this coincides with the fppf classifying stack.}
\[
\ca P \colon \on{LogPic}_{X/S}\times_S \on{LogPic}_{X/S} \to B\bb G_m^{log}.
\]
By linearity, for any $T \to S$ and any two torsion sections $a,b \colon T \to \on{LogPic}_{X/S}[n]$, the torsor $\ca P(a,b)^{\otimes n}$ admits two natural trivializations
\begin{align*}
	\on{Triv_2} & \colon \ca P(a,b)^{\otimes n} \iso \ca P(a,nb) \iso \ca P(a,0) \iso \bb G_m^{log}(T) \\
	\on{Triv_1} & \colon \ca P(a,b)^{\otimes n} \iso \ca P(na,b) \iso \ca P(0,b) \iso \bb G_m^{log}(T).
\end{align*}
The Weil pairing $\on{Weil}(a,b)$ is the difference $\on{Triv_2}\on{Triv_1}^{-1}$, which is in $\mu_n$ given that $n\on{Triv}_1=n\on{Triv}_2$ is the canonical trivialization of $\ca P(na,nb)$. For any bilinear pairing of group stacks $\Phi \colon G \times G \to H$, we denote by $\frac{\partial}{\partial p_2}\Phi$ the induced morphism $G \times BG \to BH$ (obtained by change of structure group along the maps $\Phi(a,-)$ for sections $a$ of $G$). In particular, the Weil pairing yields a bilinear map
\[
\frac{\partial}{\partial p_2}\on{Weil} \colon \on{LogPic}_{X/S}[n] \times_S B\on{LogPic}_{X/S}[n] \to B\mu_n \to B\bb G_m^{log}
\]
of klf stacks over $S$.

We turn to the proof of our earlier claim that $\ca T:=\times n \colon X' \to X$ is the universal $\on{LogPic}[n]$-torsor on $X/S$. The Weil pairing identifies $\on{LogPic}_{X/S}[n]$ with its own Cartier dual. After unwrapping definitions, the claim is equivalent to proving that the map of klf group stacks over $X$
\begin{equation}\label{toprove}
    \frac{\partial}{\partial p_2}\on{Weil}(-,\ca T) \colon \on{LogPic}_{X/S}[n] \to B\bb G_{m,X}^{log}
\end{equation}
is the natural inclusion
\begin{equation}\label{opwajgp}
    \on{LogPic}_{X/S}[n]=B\mu_n \hra B\bb G_{m,X}^{log}.
\end{equation}
We emphasize that the leftmost equality in \ref{opwajgp} holds because we are considering the classifying stack $B\mu_n$ in the klf topology. It would not hold for the fppf classifying stack. 
Now, since $\ca T$ is the universal $n$-th root of the Abel-Jacobi map $\on{AbJac} \colon X \to \on{LogPic}^0_{X/S}$, for any section $a$ in $\on{LogPic}_{X/S}[n](T)$, the definition of the Weil pairing via trivializations of $\ca P^{\otimes n}$ yields an equivalence between sections of $\frac{\partial}{\partial p_2}\on{Weil}(a,\ca T)$ and of $\ca P(a,\on{AbJac})$, which is $a$ itself. This shows that \ref{toprove} is the identity as claimed.

The fact that we started with an elliptic degeneration whose discriminant $\pi^n$ already had a $n$-th root is not essential: by an almost identical proof, the universal klf $\on{LogPic}[n]$-torsor on the universal log curve over $\o{\ca M}_{1,1}$ is obtained by pulling back the multiplication by $n$ map on the universal log Jacobian along the log Abel-Jacobi map.


\end{example}

\begin{remark}\label{Group stack}
   (Extension of $\mathbb{Z}/p\mathbb{Z}$-torsors.) Let $R$ be a discrete valuation ring of mixed characteristic $(0,p)$, with fraction field $K$. Suppose $R$ does not have a nontrivial $p$-th root of unity. Let $R'$ be the ring obtained from $R$ by adjoining such a root. Over $K':=\on{Frac}(R')$, we have a (non-canonical) isomorphism $\mu_p\simeq \mathbb{Z}/ p\mathbb{Z}$. Let $G:=\mathrm{Aut}_R(R')=\mu_p$. The \'etale $K'$-group $\mu_{p,K'} \simeq \mathbb{Z}/p\mathbb{Z}$ has a Néron model over $R'$ given by $N':=\mathbb{Z}/p\mathbb{Z}_{R'}$. By consequence, the group stack $N:=[N'/G]$ is the N\'eron model of $\mu_{p,K}$ over $R$. If $X/R$ is a log curve, we have seen in \ref{non-ext} that some $\mathbb{Z}/p\mathbb{Z}$-torsors over $X_U$ do not extend to log torsors over $X$ under $\bb Z/p\bb Z$. However, we claim that they extend to log torsors \emph{under the dual of $N$}. We keep this statement informal and unproved for now, as it requires some machinery on duality for group stacks. We refer to \cite{SBr} as a starting point for these notions. We intend to explore them further in future work.
\end{remark}

\subsubsection{$\alpha_p$-torsors on nodal curves.}\label{alpha_p-torsors}

\begin{definition}
    The finite, flat $\bb F_p$-group scheme $\alpha_p:=\Spec \bb F_p[T]/T^p$ is the subgroup of $p$-nilpotents in $\bb G_a$. We also use the name $\alpha_p$ for the restriction of $\alpha_p$ to a given scheme $S$ on which $p=0$ when we can unambiguously do so. When $S$ is a log scheme, we equip $\alpha_p$ with the $S$-strict log structure.
\end{definition}

\begin{lemma}\label{lemma:morphisms_alpha_p_to_Z/pZ_are_trivial}
    For any $\bb F_p$-scheme $S$, we have
    \[
    \on{Hom}(\alpha_p\times_{\bb F_p} S,\bb Z/p\bb Z)=\{0\}.
    \]
\end{lemma}

\begin{proof}
    Since $\bb Z/p\bb Z$ is a constant and discrete group sheaf, it suffices to show that $\alpha_p$ is geometrically connected over $\bb F_p$. This is true since $\alpha_p\times_{\bb F_p} \o{\bb F}_p$ is the spectrum of $\o{\bb F}_p[T]/T^p$, whose only prime ideal is $(T)$.
\end{proof}

\begin{lemma}\label{lemma:hom_ap_Gm_=_ap}
    The Cartier dual of $\alpha_p/\bb F_p$ is canonically isomorphic to $\alpha_p$.
\end{lemma}

\begin{proof}
    This is well-known, so we only sketch the proof. For any ring $R$ over $\bb F_p$, the $R$-points of $\alpha_p$ are the additive group of $p$-nilpotent elements of $R$ and the $R$-points of $\alpha_p^D$ are the invertibles (for multiplication) of $R[y]/y^p$. We may define a group homomorphism $\alpha_p(R) \to \alpha_p^D(R)$ by way of the truncated exponential series
    \begin{align*}
        \Phi_R \colon \alpha_p(R) & \to \alpha_p^D(R)\\
        x & \mapsto \left(y \mapsto \sum\limits_{i=0}^{p-1} \frac{(xy)^i}{i!} \right).
    \end{align*}
    The formation of $\Phi_R$ is functorial in $R$, and a functorial inverse is given by 
    \begin{align*}
        \on{Hom}(\alpha_p(R),R^\times) & \to \alpha_p(R) \\
        (y \mapsto f(y)) & \mapsto \diff{f}{y}(0).
    \end{align*}
    Thus, the $\Phi_R$'s define an isomorphism of group schemes $\alpha_p \to \alpha_p^D$.
\end{proof}

In view of \ref{theorem:Raynaud}, \ref{mainrslt} and \ref{lemma:hom_ap_Gm_=_ap}, understanding relative (fppf or klf) torsors on a log curve $X/S$ boils down to understanding morphisms from $\alpha_p$ to the (classical or logarithmic) Jacobian of $X/S$. The case of smooth curves is already difficult to describe exhaustively. We focus on the difference between smooth curves and nodal curves, as well as between fppf and klf torsors.

\begin{proposition}\label{alpha_p}
    Let $X/S$ be a log curve such that $p=0$ in $\ca O_S$. Then, the natural map
    \[
    H^1_{fppf}(X/S,\alpha_p) \to H^1_{klf}(X/S,\alpha_p)
    \]
    is an isomorphism.
\end{proposition}

\begin{proof}
    For all log scheme maps $T \to S$, $\on{TroPic}_{X/S}[p](T)$ is a power of $\bb Z/p\bb Z$. Thus, by \ref{lemma:morphisms_alpha_p_to_Z/pZ_are_trivial}, the group sheaf $\sheafhom(\alpha_p,\on{TroPic}_{X/S}[p])$ is trivial and the map
    \[
    \sheafhom(\alpha_p,\on{Pic}_{X/S}[p]) \to \sheafhom(\alpha_p,\on{LogPic}_{X/S}[p])
    \]
    is an isomorphism. We conclude using \ref{lemma:hom_ap_Gm_=_ap}, \ref{theorem:Raynaud} and \ref{mainrslt}.
\end{proof}

By \ref{alpha_p}, klf $\alpha_p$-torsors on a log curve are equivalent to fppf $\alpha_p$-torsors on the underlying prestable curve. To describe the latter, we must first discuss the $p$-torsion of Jacobians in characteristic $p$.

\begin{definition}
    Let $S$ be a scheme and $p$ a prime. A \emph{self-dual truncated Barsotti-Tate group of level $1$} over $S$, or \emph{$BT_1$} for short, is a pair $(B \to S,\phi)$ where $B/S$ is a finite, flat, commutative $S$-group killed by $p$ and $\phi$ is a $S$-isomorphism $B \iso B^D$ between $B$ and its Cartier dual. We will often abusively call $B$ a $BT_1$ and keep the identification $B=B^D$ implicit.
\end{definition}

\begin{example}
    Let $A/S$ be a principally polarized abelian variety. The polarization induces via the Weil pairing an isomorphism $A[p] \iso A[p]^D$, thereby making $A[p]$ a $BT_1$.
\end{example}

The underlying groups of $BT_1$s over algebraically closed fields of characteristic $p$ were classified independently in \cite{kraft1975kommutative} (unpublished) and in \cite{Oort2001}. We refer to \cite{PriesShortGuidepTorsion} for an introduction to this topic and to \cite{Oort2001}, \cite{Moonen2001} for a complete treatment. \cite{Oort2001} also discusses the behavior of (underlying groups of) $BT_1$s under specialization. Two important invariants of $BT_1$s are the following.

\begin{definition}\label{local_local_part_and_p_rank}
    Let $(B,\phi)$ be a $BT_1$ over an algebraically closed field $k$ of characteristic $p$. The \emph{local-local part} of $(B,\phi)$ is the largest sub-$k$-group $\ca I\subset B$ which is connected and has connected Cartier dual. If $\ca I=B$, we say $(B,\phi)$ is \emph{local-local}. The $k$-group scheme of connected components of $B/k$ is necessarily of the form $(\bb Z/p\bb Z)^r$, and we call $r$ the \emph{$p$-rank} of $(B,\phi)$.
\end{definition}

 We will need the following structure result.

\begin{proposition}\label{proposition:BT1s}
    Let $k$ be an algebraically closed field of characteristic $p$ and $(B,\phi)$ a $BT_1$ over $k$. Let $r,\ca I$ be respectively the $p$-rank and local-local part of $B$. Then
    \begin{equation}\label{eqndirectsumdecomp}
        B \simeq (\mu_p \times \bb Z/p\bb Z)^r \times \ca I.
    \end{equation}
    Furthermore, $\ca I$ is canonically autodual.
\end{proposition}

\begin{proof}
        This is known, but we prove it for convenience. Let $B^0$ be the connected component of identity in $B$. The quotient $B/B^0$ is the group of connected components of $B$. Since $k$ is algebraically closed and $B/k$ is finite, the group of $k$-points $B(k)$ is canonically isomorphic to $B/B^0$ so the quotient map $B \to B/B^0$ has a section. Pick an isomorphism $B/B^0 \iso (\bb Z/p\bb Z)^r$. Dualizing the direct sum decomposition
        \[
        B = B^0 \oplus_k (\bb Z/p\bb Z)^r
        \]
        yields
        \begin{equation*}
            B = (B^0)^D \oplus_k \mu_p^r,
        \end{equation*}
        from which we deduce a direct sum decomposition
        \begin{equation}\label{eqnpigwrn}
            B = \left(B^0 \times_B (B^0)^D\right) \oplus_k (\mu_p \times \bb Z/p\bb Z)^r.
        \end{equation}
        The sub-$k$-group $B^0 \times_B (B^0)^D$ of $B$ is autodual. All that remains to show is that it is equal to $\ca I$. We have $B^0 \times_B (B^0)^D \hra \ca I$ since $B^0 \times_B (B^0)^D$ is connected with connected Cartier dual. The other inclusion $\ca I \hra B^0 \times_B (B^0)^D$ holds since $(\bb Z/p\bb Z)^r$ is discrete and $\mu_p^r$ has discrete dual.
\end{proof}

We may now deduce the structure of the group of fppf $\alpha_p$-torsors on a prestable curve:

\begin{proposition}\label{proposition:structure_alpha_p_torsors}
    Let $\pi \colon X \to \Spec k$ be a prestable curve over an algebraically closed field of characteristic $p$. Let $X^\nu \to X$ be the normalization map. Let $r,\ca I$ be respectively the $p$-rank and local-local part of $\on{Pic}_{X^\nu/k}[p]$, which is a $BT_1$  over $k$. Let $H_1$ be the first homology group of the dual graph of $X/k$. Then
    \begin{enumerate}
        \item The normalization sequence
        \begin{equation}\label{eqn:normalisation_seq_torsionpart_pic}
       0 \to \bb \mu_p^{H_1} \to \on{Pic}_{X/k}[p] \to \on{Pic}_{X^\nu/k}[p] \to 0 
    \end{equation}
    of fppf sheaves on $\Spec k$ is exact and split.
    \item Any pair consisting of a splitting of \ref{eqn:normalisation_seq_torsionpart_pic} and an isomorphism
    \[
    \on{Pic}_{X^\nu/k}[p] \iso (\mu_p \times \bb Z/p\bb Z)^r \times \ca I
    \]
    canonically induces an isomorphism
    \[
    R^1_{fppf}\pi_*\alpha_p \iso \alpha_p^{H_1} \times \alpha_p^r \times \sheafhom_k(\alpha_p,\ca I).
    \]
    \end{enumerate}
\end{proposition}

\begin{proof}
    (2) follows from (1) by using \ref{theorem:Raynaud}, \ref{lemma:morphisms_alpha_p_to_Z/pZ_are_trivial} and \ref{lemma:hom_ap_Gm_=_ap}. The exactness of \ref{eqn:normalisation_seq_torsionpart_pic} follows from that of the sequence of fppf $k$-sheaves
    \[
    0 \to \bb G_m^{H_1} \to \on{Pic}_{X/k} \to \on{Pic}_{X^\nu/k} \to 0
    \]
    and from the surjectivity of multiplication by $p$ on $\bb G_m$ (as an fppf sheaf). Let $B_{\ca I},B_r$ be the restrictions of $\on{Pic}_{X/k}[p]$ along $\ca I \to \on{Pic}_{X^\nu/k}[p]$ and $(\mu_p \times \bb Z/p\bb Z)^r \to \on{Pic}_{X^\nu/k}[p]$ respectively. The exact sequences
    \begin{align*}
        0 \to \bb \mu_p^{H_1} \to & B_{\ca I} \to \ca I \to 0 \\
        0 \to \bb \mu_p^{H_1} \to & B_r \to (\mu_p \times \bb Z/p\bb Z)^r \to 0
    \end{align*}
    induced by \ref{eqn:normalisation_seq_torsionpart_pic} both split, the first because $\ca I$ is the largest connected subgroup of $B_{\ca I}$ with connected Cartier dual and the second because an extension of $\mu_{p,k}$ or $\bb Z/p\bb Z$ by $\mu_{p,k}$ is trivial if it is killed by $p$. Therefore, \ref{eqn:normalisation_seq_torsionpart_pic} splits as well.
\end{proof}

\begin{remark}\label{remark:alpha_p_tors_on_nodal_curve}
    In the setting of \ref{proposition:structure_alpha_p_torsors}, describing all the possible local-local $BT_1$ groups $\ca I$ can be difficult, so we do not aim to give a full description of $\alpha_p$-torsors on prestable curves. We refer the interested reader to \cite{PriesShortGuidepTorsion} and \cite{Oort2001}. It seems known to the experts that $\ca I$ always admits a filtration whose successive quotients are direct sums of copies of $\alpha_p$, so we include \ref{proposition:hom_ap_ap_=_Ga} in this discussion. 
\end{remark}

\begin{lemma}\label{proposition:hom_ap_ap_=_Ga}
    There is a canonical isomorphism $\sheafhom(\alpha_p,\alpha_p)=\bb G_a$.
\end{lemma}

\begin{proof}
    Let $R$ be a ring over $\bb F_p$. Any $f \colon \Spec R[x]/x^p \to \Spec R[y]/y^p$ is characterized by the image $f^\#(y)$ of $y$ in $R[x]/x^p$, and $f$ is a group homomorphism if and only if $f^\#(y)$ is of the form $ax$ for some $a\in R$. This gives a bijection $\sheafhom(\alpha_p,\alpha_p)(R) \iso \bb G_a(R)$. This bijection is functorial, hence defines a canonical isomorphism $\sheafhom(\alpha_p,\alpha_p)=\bb G_a$.
\end{proof}

\begin{remark}
    Let $s=(\Spec k,\o M)$ be a log geometric point such that $k$ is algebraically closed and of characteristic $p$. Let $X/s$ be a log curve such that the tropical length (in $\o M$) of every node of $X$ is a multiple of $p$. Let $X^\nu$ be the normalization of $\ul X$ and $r,\ca I$ be the $p$-rank, resp. local-local part of $\on{Pic}_{X^\nu/k}[p]$. The log Weil pairing discussed in \ref{example:univ_torsor_on_nodal_cubic} makes $\on{LogPic}_{X^\nu/k}[p]$ a $BT_1$ with local-local part $\ca I$ and $p$-rank $r+h_1$ where $h_1$ is the first Betti number of the dual graph of $X/k$. Together, \ref{alpha_p}, \ref{proposition:BT1s} and \ref{proposition:structure_alpha_p_torsors} show that the group scheme of klf (equivalently, fppf) $\alpha_p$-torsors on $X/s$ is explicitly determined by its $BT_1$. In particular, there always exists a \emph{smooth} $k$-curve with the same group of isomorphism classes of $\alpha_p$-torsors as $X/s$.
\end{remark}

\begin{remark}
    Let $\ca A_g$ be the moduli stack of $g$-dimensional, principally polarized abelian varieties. Let $\ca U \to \ca A_g$ be the universal abelian variety. Oort shows in \cite{Oort2001} that $\ca A_g$ is stratified by the loci on which the $BT_1$ $\ca U[p]$ has constant isomorphism class. We only discussed Jacobians here, but one should also expect the Oort stratification to extend to a similar stratification of (a sufficient root stack of) the moduli space $\ca A_g^{log}$ of principally polarized log abelian varieties. Let $\partial^{\on{Oort}}$ be the boundary divisor of this stratification, i.e. the locus where the geometric fibres of $\ca U[p]$ are not isomorphic to $(\mu_p \times \bb Z/p\bb Z)^g$. The torus/tropical part of $\ca U$ only contributes to the $p$-rank of $\ca U[p]$, so the boundary divisors $\ca A_g^{log} \setminus \ca A_g$ and $\partial^{\on{Oort}}$ intersect transversally in $\ca A_g^{log}$.
\end{remark}

\bibliographystyle{alpha} 
\bibliography{prebib}

\bigskip

\noindent

\bigskip

\noindent
Sara MEHIDI, {\sc Mathematical Institute, Utrecht University, Hans Freudenthalgebouw, Budapestlaan 6, 3584 CD Utrecht, Netherlands.\\
Email address: {\tt s.mehidi@uu.nl}\\
Thibault POIRET,
{\sc  School of Mathematics and Statistics, University of St Andrews, Mathematical Institute North Haugh St Andrews, KY16 9SS, United Kingdom.} \\
Email address: {\tt tdp1@st-andrews.ac.uk}

\end{document}